\documentclass[preprint,12pt,3p]{elsarticle}

\usepackage{amsmath,amssymb,amsthm}
\usepackage{mathrsfs}
\usepackage{enumerate}
\usepackage{hyperref}
\usepackage{xcolor}

\newtheorem{theorem}{\bf Theorem}[section]
\newtheorem{lemma}[theorem]{\bf Lemma}
\newtheorem{proposition}[theorem]{\bf Proposition}
\newtheorem{corollary}[theorem]{\bf Corollary}
\theoremstyle{definition}
\newtheorem{definition}[theorem]{\bf Definition}

\journal{Journal of LMS}

\begin{document}

% ------------------- 标题 & 作者 -------------------
\begin{frontmatter}

\title{Sharp stability of $\Delta u - u + |u|^{p-1}u$ near a finite sum\\ of ground states}

\author[whu]{Hua Chen\corref{cor1}}
\ead{chenhua@whu.edu.cn}

\author[whu]{Yun Lu Fan}
\ead{yunlufan@whu.edu.cn}

\author[hnu]{Xin Liao}
\ead{xin\_liao@whu.edu.cn}

\cortext[cor1]{Corresponding author}

\address[whu]{School of Mathematics and Statistics, Wuhan University, Wuhan 430072, China}
\address[hnu]{School of Mathematics and Statistics, Hunan Normal University, Changsha 410012, China}

% ------------------- 摘要 & 关键词 -------------------
\begin{abstract}
Let $ Q $ denote the unique radial solution to the equation
$$
\Delta Q - Q + |Q|^{p-1} Q = 0 \quad \text{in} \ \mathbb{R}^d, \quad Q > 0, \quad Q \in H^1(\mathbb{R}^d),
$$
where $1<p< 2^* - 1$. We establish sharp quantitative stability estimates near finite sums of ground states for this equation. Specifically, for a set $ \{\widetilde{y}_k\}_{k=1}^m \subset \mathbb{R}^d $, where $ \min\limits_{i \neq j} |\widetilde{y}_i - \widetilde{y}_j| $ is sufficiently large, we show that if $ u $ is close to $ \sum_{k=1}^m Q(\cdot + \widetilde{y}_k) $ in $H^1(\mathbb{R}^d)$, then the following inequality holds:
$$
\inf_{\{y_k\} \subset \mathbb{R}^d} \| u - \sum_{k=1}^m Q(\cdot + y_k) \|_{H^1} \leq C(m,d, p) F_{d,p}\left( \|\Delta u - u + |u|^{p-1} u\|_{H^{-1}} \right).
$$
The function $F_{d,p}:\mathbb{R}^+\to\mathbb{R}^+$ is defined for $s>0$ via
$$
  F_{d,p}(s)=
  \begin{cases}
  s, & \mbox{if } p>2 , \mbox{or } p=2, d=4,5 \\
  (|\ln s| +1)^{\frac{1}{2}}s, & \mbox{if } p=2, d=1 ,\\
  |\psi(s)|^{-\frac{1}{4}}e^{- \psi(s)}, & \mbox{if } p=2, d=2,\\
   |\psi(s)|^{-1}e^{-\psi(s)}\ln^{\frac{1}{2}} (\psi(s)+2), & \mbox{if } p=2, d=3,\\
   |\psi(s)|^{(\frac{1}{4}-\frac{p}{2})(d-1)}e^{-\frac{p}{2}\psi(s)}, & \mbox{if }  1<p<2, d\geq 1
        \end{cases}
$$
where  $\psi$ is the inverse of the monotone function
$\phi (t)=t^{-(d-1)/2}e^{-t}$ in $(0, +\infty)$,
and we extend $F_{d,p}$ continuously at $0$ by setting $F_{d,p}(0)=0$.

Furthermore, by employing the co-compact method, we derive a global stability result from the above local one when $ d = 1 $ or when $ u $ is non-negative. In the complex-valued case, the analysis becomes more intricate, and may depend on the phase difference. Finally, we give some applications to the nonlinear Schrödinger equation.
\end{abstract}


\end{frontmatter}

% ------------------- 正文开始 -------------------
\section{Introduction}
$\mathbf{Notations}$:
Throughout the entire paper, the symbol
$C$ will denote a positive constant, which may vary throughout the discussion. The term $o_{R}(1)$  represents an infinitesimal quantity with respect to $R$ as $R\to+\infty$. Furthermore, we define  $o_{R}(f)=o_{R}(1)f$. 
We say  $g=O(f)$ if $|g|\leq C|f|$ for some constant $C$.  
The notation
$X\lesssim Y$ means that there exists a positive constant $C$ such that $X\leq CY$. Similarly, $X\gtrsim Y$  means that there exists a positive constant
$C$ such that  $X\geq CY$.
We say $X\approx Y$  if both $X \lesssim Y$ and $X\gtrsim Y$ hold.
Additionally, $u^{+}$ and $u^{-}$ denote the positive and negative parts of $u$, respectively.

\subsection{Motivation}
The ground state $Q$ is  the unique radial solution to the equation, 
\[\Delta Q-Q+Q^p=0, ~~ Q>0, ~Q\in H^{1}(\mathbb{R}^d).\]
This equation arises not only in the context of Bose-Einstein condensates in physics (\cite{BEC}), but is also closely related to pattern formation in mathematical biology (\cite{Wei2013}). Moreover,
the function  $Q$ attains the best constant in both the nonhomogeneous Sobolev inequality and the Gagliardo–Nirenberg inequality, making it a fundamental object in the study of these functional inequalities.

Ground states also play a crucial role in the theory of nonlinear dispersive equations, such as
 the  Nonlinear Schrödinger (NLS) and Korteweg–de Vries (KdV) equations. For instance, the function
 $e^{it}Q$ is a standing wave solution to the NLS equation:
\[ i\partial_t u=-\Delta u-|u|^{p-1} u,\]
whereas $Q(x-t)$ is a traveling wave solution to the KdV equation:
\[ -\partial_tu=\partial_{xxx}u+pu^{p-1}\partial_{x}u.\]
Both the NLS and KdV equations are invariant under various symmetric transformations. For the NLS, the family of 1-soliton solutions can be expressed as: $$\lambda^{\frac{p-1}{2}} e^{i\lambda^2t+i\theta+iv x-iv^2t}Q(\lambda( x+y-2vt)),$$
where $\lambda \in \mathbb{R}$ is the  scaling parameter,  $\theta \in [0, 2\pi)$ is the phase  rotation parameter, $y\in \mathbb{R}$ is the translation parameter, and $v\in \mathbb{R}$ is the speed parameter.
 For the KdV equation, the 1-soliton solutions take the form:
  \[v^{\frac{p-1}{2}}  Q(v (x-vt+y)),\]
  with  the translation parameter
 $y\in \mathbb{R}$ and speed parameter $v\in \mathbb{R}$.
 %Multi-soliton solutions, roughly speaking, refer to solutions to the NLS or KdV equations that scatter to a sum of individual 1-soliton solutions, as discussed in Nguyen \cite{Nguyen2019}.
The stability of ground states is crucial in scattering and soliton theories for the NLS and KdV equations. For insights, see Tao \cite{Tao}  the references therein. Notably, Martel, Merle, and Tsai \cite{Martel2002, Martel2006} analyzed the  $H^1$
orbital stability of sums of 1-soliton solutions with different speed parameters. More recently, the
$H^s$ orbital stability of multi-soliton solutions has been investigated by Koch and Tataru \cite{Koch2024}.

Over the past few years, the stability of functional inequalities and elliptic equations has attracted significant attention (cf. \cite{Bonforte,  Ciraolo, Wei, FigalliNeumayer, FigalliZhang, Frank, FuscoMaggiPratelli, LiuZhangZou, Wei2022}).
We illustrate the concept of stability using the Sobolev inequality as an example.  The homogeneous Sobolev inequality states that for $d\geq 3$ and any function $u\in \dot{H}^{1}(\mathbb{R}^d)$, the following holds:
\[S\|u\|_{L^{2^*}}^2\leq \|\nabla u\|_{L^2}^2\]
where $2^*=\frac{2d}{d-2}$ is the critical Sobolev exponent, and $S$ is the best constant. The extremal functions achieving equality (up to a multiplicative constant) are known as Aubin–Talenti (\cite{Aubin,Talenti}) bubbles: 
$$U[\lambda, y](x)=(d(d-2))^{{\frac{d-2}{4}}}(\frac{\lambda}{1+\lambda^2|x-y|^2})^{\frac{d-2}{2}}.$$
These bubbles are all the positive solutions to the following well-known Yamabe equation in the whole $\mathbb{R}^d$:
\begin{equation}\label{yamabe}
-\Delta u=|u|^{2^*-2}u~~\mbox{in }  \mathbb{R}^d.
\end{equation}

The stability of the Sobolev inequality concerns the following question:
If a function 
$u$ nearly attains equality in the Sobolev inequality, must it be quantitatively close to one of the extremal functions?
Bianchi and Egnell \cite{Bianchi} provided a precise answer by proving the quantitative stability estimate: 
\begin{equation}\label{1.2}
 \inf\limits_{y,\lambda ,c}\|\nabla( u- cU[\lambda, y])\|_{L^2} \leq C(d)  (\|\nabla u\|_{L^2}^2 -S\|u\|_{L^{2^*}}^2).
\end{equation}

A more challenging problem is  to consider the stability of  equation \eqref{yamabe}: whether a function $u$ that almost solves \eqref{yamabe}  must be quantitatively close to
Aubin–Talenti bubbles?   Struwe's celebrated global compactness result (cf. \cite{Struwe}) yields that:
 \begin{proposition}
  Let $\{u_k\} \subset \dot{H}^{1}(\mathbb{R}^d)$ be a sequence of nonnegative functions  such that $$(m-\frac{1}{2})S^{\frac{d}{2}}\leq \|u_k\|_{ \dot{H}^{1}}
  \leq (m+\frac{1}{2})S^{\frac{d}{2}} $$  for some positive integer $m$, and
  \[ \|\Delta u_k+|u_k|^{2^*-2}u_k   \|_{\dot{H}^{-1}} \to 0 ~\mbox{as } k\to+\infty. \]
  Then there exists a sequence $\{\lambda_i^{(k)}, y_i^{(k)}\}$ such that
  \[ \|u_k-\sum_{i=1}^{m} U[\lambda_i^{(k)}, y_i^{(k)}] \|_{\dot{H}^{1}} \to 0 ~\mbox{as } k\to+\infty. \]
  Furthermore, for $i\neq j$, we have
  \[ \varepsilon_{ij}^{(k)}:=\min \{ \frac{\lambda_i^{(k)}}{\lambda_j^{(k)}},  \frac{\lambda_j^{(k)}}{\lambda_i^{(k)}}, \frac{1}{\lambda_i^{(k)}\lambda_j^{(k)} |y_i^{(k)}-y_j^{(k)}|^2} \} \to 0 ~\mbox{as } k\to+\infty.   \]
\end{proposition}

Motivated by this, it is natural to compare the asymptotic rates of convergence of $$\|\Delta u+|u|^{2^*-2}u   \|_{\dot{H}^{-1}}, ~\|u-\sum_{i=1}^{m} U[\lambda_i, y_i] \|_{\dot{H}^{1}} \mbox{ and } \varepsilon_{ij}$$  as they tend to zero.

\begin{definition}
 For a finite collection of bubbles 
 $\{U[y_i,\lambda_i ]\},$ 
 define
$$ \varepsilon_{ij}:=\min \{ \frac{\lambda_i}{\lambda_j},  \frac{\lambda_j}{\lambda_i}, \frac{1}{\lambda_i\lambda_j |y_i-y_j|^2} \}.$$
 We say these bubbles are  $\delta$-weakly interacting if $$\varepsilon:=\max\{\varepsilon_{ij}\}\leq \delta.$$
 \end{definition}
 Figalli and Glaudo \cite{Figalli2020} were the first to study the stability of the Yamabe equation near a finite sum of weakly interacting Talenti bubbles. They established a sharp quantitative  estimate for  space dimensions $3\leq d \leq 5$. Later, Wei, Deng, and Sun \cite{Wei} extended these results by proving sharp quantitative stability estimates for $d\geq 6$.

We summarize their results as follows:

 \begin{theorem}
For each $m$ and $u\in \dot{H}^{1}(\mathbb{R}^d)$, there exists constant $\widetilde{\delta}=\widetilde{\delta}(n, m)$ such that for any $\delta \in (0, \widetilde{\delta})$,
if there exists a family of $\delta$-weakly interaction  bubbles  $\{U[\widetilde{y_k},\widetilde{\lambda_k}]\}_{k=1}^{m}$  satisfying
\[\|u-\sum_{k=1}^{m}U[\widetilde{y_k},\widetilde{\lambda_k}] \|_{\dot{H}^{1}} <\delta,\]
then  there exists  a finite collection of bubbles $\{U[y_i,\lambda_i ]\}_{i=1}^{m}$, such that 
 the following sharp stability estimate  
\begin{equation}
\|u-\sum_{k=1}^{m}U[y_k,\lambda_k] \|_{\dot{H}^{1}} \lesssim
\begin{cases}
 \Gamma(u), & \mbox{if } 3\leq d\leq 5 \\
  \Gamma(u)|\log \Gamma(u) |^{\frac{1}{2}}, & \mbox{if } d=6 \\
  \Gamma(u)^{\frac{d+2}{2(d-2)}}, & \mbox{if } d>6.
\end{cases}
\end{equation}
holds for
$\Gamma(u):=\|\Delta u+|u|^{2^*-2}u\|_{\dot{H}^{-1}}.$ Furthermore, we have
 $$\varepsilon^{\frac{d-2}{2}} \lesssim \Gamma(u).$$
\end{theorem}

\subsection{Main results}
By using the co-compact method, Cao-Peng-Yan \cite[Theorem 1.4.1]{Cao2021} obtained the following result:
 \begin{theorem}\label{theorem1.3} Let $1< p<2^*-1$.
 Let $\{u_n\}$ be a bounded sequence in  $H^{1}$ such that $$\|\Delta u_n-u_n+|u_n|^{p-1}u_n\|_{H^{-1}}\to 0.$$ (i.e. $\{u_n\}$ is a Palais–Smale  sequence.)   Then up to a subsequence,
  there exists an integer  $m$, a sequence $\{v^{k}\}_{k=1}^{m}$ with $$\Delta v^{k}-v^{k}+|v^{k}|^{p-1}v^{k}=0,$$ and a sequence of points $\{y_{n,k}\}$ such that
   \[\min_{i\neq j}|y_{n,i}-y_{n,j}| \to +\infty, ~~ u_n \to \sum_{k=1}^{m} v^{k}(\cdot+y_{n,k})\mbox{ as } n\to+\infty, \] 
and  it holds that
   $$\|u_{n}\|_{H^1}^2 \to \sum_{k=1}^{m} \|v^{k}\|_{H^1}^2 \mbox{ as } n\to+\infty .$$
 \end{theorem}
 
 Thus,
inspired by   \cite{Wei} and \cite{Figalli2020}, we consider the $H^1$ stability  of the scalar field equation:
 \begin{equation}\label{cfe}\Delta u-u+|u|^{p-1}u \text{~~in~~} \mathbb{R}^d,\end{equation}
near a finite sum of ground states when $1< p<2^*-1$.  
For the real-valued case, we present our main result as follows:

\begin{theorem}\label{theorem1.2}
Let $1< p<2^*-1$.
For each $m\geq 1$,  there exists a large constant $\bar{R}=\bar{R}(m, d, p)>10$, such that if
   $u\in H^{1}(\mathbb{R}^d)$ satisfies
   $$\|u-\sum_{k=1}^{m}Q(\cdot+\widetilde{y}_k)\|_{H^1}\leq \frac{1}{\bar{R}},~\text{for~some~} \{\widetilde{y}_k\}\subset \mathbb{R}^d \text{~with~}  \min_{i\neq j}|\widetilde{y}_i-\widetilde{y}_j|\geq \bar{R},$$
then there exists $\{y_k\}_{k=1}^{m} \subset \mathbb{R}^d$ with $R:=\min_{i\neq j}|y_i-y_j|$ sufficiently large, such that
the following  estimation:
\begin{equation}\label{main}
\|u-\sum_{k=1}^{m}Q(\cdot+y_k)\|_{H^1} \leq C(m, d, p) F_{d, p}( \|\Delta u-u+|u|^{p-1}u\|_{H^{-1}})
\end{equation}
holds for some  constant $C(m,d, p)$ depending on $m, d$ and $p$.
Here,   the function $F_{d,p}:\mathbb{R}^+\to\mathbb{R}^+$ is defined for $s>0$ via
\begin{equation}\label{mainnn}
  F_{d,p}(s)=
  \begin{cases}
  s, & \mbox{if } p>2 , \mbox{or } p=2, d=4,5 \\
  (|\ln s| +1)^{\frac{1}{2}}s, & \mbox{if } p=2, d=1 ,\\
  |\psi(s)|^{-\frac{1}{4}}e^{- \psi(s)}, & \mbox{if } p=2, d=2,\\
   |\psi(s)|^{-1}e^{-\psi(s)}\ln^{\frac{1}{2}} (\psi(s)+2), & \mbox{if } p=2, d=3,\\
   |\psi(s)|^{(\frac{1}{4}-\frac{p}{2})(d-1)}e^{-\frac{p}{2}\psi(s)}, & \mbox{if }  1<p<2, d\geq 1
        \end{cases}
\end{equation}
where  $\psi$ is the inverse of the monotone function
$\phi (t)=t^{-(d-1)/2}e^{-t}$ in $(0, +\infty)$,
and we extend $F_{d,p}$ continuously at $0$ by setting $F_{d,p}(0)=0$.
Furthermore, we have
\[ R^{-\frac{d-1}{2}} e^{-R} \lesssim \|\Delta u-u+|u|^{p-1}u\|_{H^{-1}}, \quad \forall~ i\neq j.\]
\end{theorem}
%The function $F_{d,p}$ arises from  tedious computations in Section 3.1, that is $$\|(Q+Q(\cdot+Re_1))^p-Q^p-Q(\cdot+Re_1)^p\|_{L^2}\approx F_{d,p}(R^{-\frac{d-1}{2}} e^{-R}).$$

We will construct examples to demonstrate the sharpness of estimate \eqref{main}-\eqref{mainnn}. 
The sharpness of the result for \( p > 2 \) and \( p = 2, d > 3 \) is straightforward, as one can simply consider \( u = (1 + \varepsilon)Q \). Therefore, we only focus on the case \( 1\leq p < 2 \) 
and \( p = 2, d = 1, 2, 3 \). 
In particular, we have the following result:

\begin{theorem}\label{sharp}
  The stability estimate \eqref{main}-\eqref{mainnn} is sharp when \( 1< p < 2 \) 
and \( p = 2, d = 1, 2, 3 \), in the sense that for each $m\geq 2$ and
  for all $R\geq \bar{R}(m,p)$ large enough, there exists $u=\sum\limits_{k=1}^{m} Q(\cdot + y_k)+\rho$ such that 
  \[  \min\limits_{i\neq j} |y_i-y_j|\geq R ,~~ \|\rho\|_{H^1}\leq \frac{1}{R} \]
and the following holds:
  \begin{equation}\label{mainsharp}
 \|u - \sum_{k=1}^{m} Q(\cdot + y_k)\|_{H^1} \gtrsim F_{d, p}( \|\Delta u - u + |u|^{p-1} u\|_{H^{-1}}).
  \end{equation}
  Moreover, the function \( u \) can be chosen to be non-negative.
\end{theorem}

 When $d=1$, Wei and Wu \cite[Section 5]{Wei2022} actually proved the same sharp  stability results for a  slightly different equation.

The function $ F_{d, p}$ arises from tedious computations of the interaction between different ground states.
In one dimension, the ground state can be explicitly expressed, and the integral calculations are straightforward. In contrast, the interactions between ground states in higher dimensions are more challenging to compute. As shown in Section 3.1, the asymptotic behavior depends not only on the dimension $d$ but also on the parameter $p$. 

In fact,
under the assumption of  Theorem \ref{theorem1.2},  since $u$ is close to $\sum_{k=1}^{m}Q(\cdot+\widetilde{y}_k)$, 
 the term $$ 
\inf\limits_{ \{y_i\}_{i=1}^{m}\subset \mathbb{R}^d} \|u - \sum\limits_{i=1}^{m} Q(\cdot + y_i)\|_{H^1}
$$
can be achieved by some
\[
\sigma := \sum_{i=1}^{m}  Q_i, \quad Q_i := Q(\cdot + y_i),
\]
with
$$
R:= \min_{i \neq j} |y_i - y_j| \text{~~large~enough~and } \max_{i} |y_i-\widetilde{y}_i| \text{~~small~enough}.
$$
If we write $u=\sigma+\rho$, by
taking derivative with the function $$G(y_1,\ldots, y_m)=\|u - \sum\limits_{i=1}^{m} Q(\cdot + y_i)\|_{H^1}^2,$$ we obtain
the following  orthogonality condition:
\begin{equation}
 \int_{\mathbb{R}^d}  (\rho \nabla Q_i + \nabla \rho \cdot \nabla^2 Q_i) dx  =p\int_{\mathbb{R}^d}   \rho   Q_i^{p-1} \nabla Q_i  dx=0,\quad \forall i.
\end{equation}
Moreover, 
by direct calculation, we obtain
\begin{equation}\label{rho}
\begin{aligned}
(-\Delta+1) \rho-p\sigma^{p-1} \rho
&= f+h +  N(\rho),
\end{aligned}
\end{equation}
where \begin{equation}\label{fh}
f:=(\sum\limits_{i=1}^{m}  Q_i )^p -\sum\limits_{i=1}^{m}  Q_i^p,\quad  h:=-\Delta u+u -|u|^{p-1}u
\end{equation}
 and
  \begin{align}\label{N}
  N(\rho):=|\sigma+\rho|^{p-1}(\sigma+\rho) -\sigma^p -p\sigma^{p-1} \rho.
  \end{align}
Thus, to provide a quantitative estimate for  $\|\rho\|_{H^1}$, we need to estimate the interaction term $f$.  The estimates in Section 3 yield that 
 $$\|f\|_{H^{-1}}\approx \|f\|_{L^2}\approx  F_{d,p}(R^{-\frac{d-1}{2}} e^{-R}).$$ 
Unlike the complicated pointwise estimates used in  \cite{Wei} and \cite[Section 5]{Wei2022}, we derive the following important estimation simply from the orthogonality condition and the non-degeneracy of ground states:  $$\|\rho\|_{H^1}\lesssim \|f\|_{H^{-1}} +\|h\|_{H^{-1}},$$
as shown in Proposition \ref{p1}.  On the other hand,  by testing the equation for $\rho$ with $\nabla Q_i$ we deduce another important estimate in Proposition \ref{prop3.}:
\[ R^{-\frac{d-1}{2}}e^{-R}\lesssim \|h\|_{H^{-1}}+\|\rho\|_{H^1}^{2}+ e^{-(1-\min\{1,\frac{p(p-1)}{4} \}) R}\|\rho\|_{H^1}^{\min\{p-1,1\}}\]
 Combining these two important estimates, we complete the proof of Theorem \ref{theorem1.2}.
 
However, since the estimate we need, $\int_{\mathbb{R}^d}f\nabla Q_i$, depends on the  position of  $y_i$ among these points $\{y_k\}$, our Lemma \ref{point} 
 is necessary, although it is trivial in one dimensional case.

\subsection{Structure of the paper}
The paper is organized as follows: \begin{enumerate} [$\bullet$] \item Section 2: Preliminaries and useful lemmas.

\item Section 3: The real-valued case. \begin{itemize} \item Section 3.1: Detailed calculations of the interaction between two ground states. 
\item Section 3.2: Derivation of the two main estimates and completion of the proof of Theorem \ref{theorem1.2}. 
    \item Section 3.3: Sharp examples to prove Theorem \ref{sharp}.
     \item Section 3.4: Proofs of Corollaries \ref{corollary1.1} and \ref{corollary1.2} based on Theorem \ref{theorem1.2}. \end{itemize}

\item Section 4: The complex-valued case. 
\begin{itemize}
 \item Section 4.1: In the complex-valued case, even for the simplest case 
$p=3$, we still need to impose restrictions on the phase difference, and it remains unclear whether this restriction can be removed, as demonstrated in Theorem \ref{theorem1.4}. 
\item Section 4.2: Focus on the single ground state case and  the application to the NLS equations, including the proof of Theorem \ref{theorem1.5} and Corollary \ref{corollary1.3}. \end{itemize} \end{enumerate}

\section{Preliminaries}
In this section, we first introduce some notations and basic properties of the ground state $Q$.
Let $H^1(\mathbb{R}^d)$  denote the  real-valued non-homogeneous Sobolev space, equipped with the inner product
$$( u, v)_{H^1}=\int_{\mathbb{R}^d}(\nabla u \cdot \nabla v+  uv) dx. $$
For complex-valued functions, we use the notation  $H^1(\mathbb{R}^d;\mathbb{C})$,  and the inner product is defined as
$$( u, v )_{H^1}=\text{Re} \int_{\mathbb{R}^d}(\nabla u \cdot \nabla \bar{v}+  u \bar{v}) dx. $$
The corresponding norm is given by : $$\|u\|_{H^1}:=\sqrt{ ( u, u )_{H^1}}.$$
We denote $H^{-1}$  as the dual space of $H^1$, with the dual pairing
$$(u,v)_{H^1,H^{-1}}:= \text{Re} \int_{\mathbb{R}^d} u\bar{v}dx, ~~\forall u\in H^1, v\in H^{-1} .$$
The critical Sobolev exponent is defined as
\begin{equation}
2^*:=
 \begin{cases}
   +\infty, & \mbox{if } d<3 \\
   \frac{2d}{d-2}, & \mbox{if } d\geq 3.
 \end{cases}
\end{equation}
The following property of the ground state is well known (  cf.  \cite[Section 1.3]{Nguyen2019}, \cite{Kwong1989}, \cite[Theorem 2]{GNN}, \cite[Chapter 4.2]{Ambrosetti}, \cite[Chapter 2]{Cao2021}, \cite[Appendix]{Wei2013}):
\begin{proposition}
Let $d\geq 1$ and $1<p<2^*-1$.  There exists a unique radial solution $Q$ to the equation
\[
\Delta Q - Q + Q^p = 0, \quad \text{in } \mathbb{R}^d, \quad Q > 0, \quad Q \in H^1(\mathbb{R}^d).
\]

The solution $Q$ satisfies the following properties:
  \begin{enumerate}
    \item For $|x|>1$, there exists a constant $c_Q>0$ such that
\begin{equation}\label{infi}
\begin{aligned}
     |Q(x) - c_Q |x|^{-\frac{d-1}{2}} e^{-|x|}|+ &|\nabla Q(x) - c_Q \frac{x}{|x|} |x|^{-\frac{d-1}{2}} e^{-|x|}| \\
    & \lesssim |x|^{-\frac{d+1}{2}} e^{-|x|}.
\end{aligned}
\end{equation}

    \item  There exists $\kappa>0$ such that
    the eigenvalue problem
  \begin{equation*}
    (-\Delta+1)\varphi=\lambda Q^{p-1} \varphi, \quad \varphi \in H^1(\mathbb{R}^d),
  \end{equation*}
admits a sequence of increasing eigenvalues
$\{\lambda_k\}$ such that $$\lambda_1=1, \lambda_2=\ldots=\lambda_{d+1}=p, \lambda_{d+2}=p+\kappa .$$
Here, $Q$ is the eigenfunction corresponding to eigenvalue $1$, and $\partial_{1}Q, \ldots, \partial_{d}Q$ are the eigenfunctions corresponding to the eigenvalue $p$. Therefore, for  $u\in H^{1}(\mathbb{R}^d)$, we have
\begin{equation}\label{gap}
\begin{aligned}
    \|u\|^2_{H^1} &\geq (p+\kappa)\int_{\mathbb{R}^d} Q^{p-1} u^2 \, dx \\
    &\quad - \frac{(p+\kappa-1)(u, Q)_{H^1}^2}{\|Q\|^2_{H^1}} 
    - \frac{\kappa}{p} \sum_{i=1}^{d} \frac{(u, \partial_i Q)_{H^1}^2}{\|\partial_i Q\|^2_{H^1}}.
\end{aligned}
\end{equation}
  \end{enumerate}
\end{proposition}

We conclude this section with the following two lemmas.
\begin{lemma}\label{lem3.1}
For $a_k\geq 0, k= 1, \ldots, m$, we have
\begin{equation*}\label{ppp}
\Big{|} \left(\sum_{k=1}^{m} a_k\right)^p - \sum_{k=1}^{m} a_k^p - p\sum\limits_{\substack{i,j\\i\neq j}} a_i^{p-1} a_j \Big{|} \lesssim
\begin{cases}
  \sum\limits_{\substack{i,j \\ i\neq j}} a_i^{\frac{p}{2}} a_j^{\frac{p}{2}} +   
  \sum\limits_{\substack{i,j,k\\ i\neq j , j\neq k, i\neq k}} a_i^{\frac{p-1}{2}} a_j^{\frac{p-1}{2}} a_k, & \mbox{if } 2< p\leq 3, \\[10pt]
  \sum\limits_{\substack{i,j \\ i\neq j}} a_i^{\frac{p}{2}} a_j^{\frac{p}{2}} +  
  \sum\limits_{\substack{i,j,k\\i\neq j , i\neq k}} a_i^{p-2} a_j a_k, & \mbox{if } p>3.
\end{cases}
\end{equation*}
Moreover,  for $1\leq p\leq 2 $, we have
\begin{equation*}
 \Big{|} \left(\sum_{k=1}^{m} a_k\right)^p - \sum_{k=1}^{m} a_k^p  \Big{|} \lesssim  \sum\limits_{\substack{i,j \\ i\neq j}} (a_i+a_j)^p-a_i^p-a_j^p.
\end{equation*}
\end{lemma}
\begin{proof}
The first inequality follows from an induction of  the following two elementary estimates for positive constants $a, b$ (cf. Cao-Peng-Yan \cite[Lemma 6.1.2]{Cao2021}):
  \begin{equation*}
  |(a+b)^{p}-a^p-b^p-pa^{p-1}b-pb^{p-1}a|\lesssim
    \begin{cases}
     a^{\frac{p}{2}}  b^{\frac{p}{2}}  , & \mbox{if } 2< p\leq 3  \\
      a^{\frac{p}{2}}  b^{\frac{p}{2}}+a^{p-2}b^2+ b^{p-2}a^2, &  \mbox{if } p>3,
    \end{cases}
     \end{equation*}
  and
      \begin{equation*}
  |(a+b)^{p}-a^p-b^p|\lesssim
    \begin{cases}
     a^{\frac{p}{2}}  b^{\frac{p}{2}}  , & \mbox{if } 1\leq p\leq 2  \\
      a^{p-1}b+ b^{p-1}a, &  \mbox{if } p>2.
    \end{cases}
     \end{equation*}
    The second inequality follows directly from Lemma A.6 in \cite{Wei}.
\end{proof}

 \begin{lemma}\label{point}
 Given a set of distinct points  $\{x_{k}\}_{k=0}^{m} \subset  \mathbb{R}^d$, there exists a constant  $c=c(m,d)>0$, depending only on the dimension $d$ and the number of points $m$, such that after applying a suitable rotation, translation, and reordering of the points  $\{x_{k}\}$, the following conditions hold:

$$x_0=0, ~~x_{k}^{1}>0,  ~~~ \forall k\geq 1~~ \text{  and  }~~ \frac{x_{k}^{1}}{|x_{k}|}>c.$$

Here $x_{k}^{1}$ denotes the first coordinate of $x_k$.
\end{lemma}
\begin{proof}
We divide the proof into two cases:

\textbf{Case 1: All points lie on a line.}

If all the points \( \{x_k\} \) are collinear, the result is immediate.

\textbf{Case 2: The points are not all collinear.}

In \( \mathbb{R}^d \) with \( d > 1 \), the lines through the origin can be parameterized by the angles \( \theta = (\theta_1, \dots, \theta_{d-1}) \in [0, \pi)^{d-1} \). We denote the corresponding line by \( L_\theta \). For a given non-zero vector \( \vec{a} \), we define \( \vec{a}_\theta \) as the projection of \( \vec{a} \) onto the line \( L_\theta \).

For a small \( \delta > 0 \), define the set \( A_{\vec{a}, \delta} \) as:
\[
A_{\vec{a}, \delta} := \left\{ \theta \in [0, \pi)^{d-1} : \frac{|\vec{a}_\theta|}{|\vec{a}|} \leq \delta \right\}.
\]
We then define the measure of this set:
\[
A(\delta) := |A_{\vec{a}, \delta}|.
\]
The quantity \( A(\delta) \) is independent of \( \vec{a} \) and satisfies:
\[
A(\delta) \to 0 \quad \text{as} \quad \delta \to 0.
\]

Now, fix \( \delta > 0 \) small enough such that \( m^2 A(\delta) < \frac{1}{1000} \). Under this condition, there must exist a line \( L_{\tau} \) such that \( \tau \notin \bigcup_{i,j} A_{\vec{x_i} - \vec{x_j}, \delta} \).

If we project all the points \( \{x_k\}_{k=0}^m \) onto the line \( L_\tau \), it holds that:
\[
\frac{|(\vec{x_i} - \vec{x_j})_\tau|}{|\vec{x_i} - \vec{x_j}|} > \delta
\]
for all distinct pairs \( i \neq j \).

This allows us to apply the result from \textbf{Case 1} (where the points are collinear) after projecting the points onto \( L_\tau \). Thus, the conditions in the lemma hold, and the proof is complete.

\end{proof}

\section{The real-valued case}
\subsection{Interaction Integral Estimates}
As a consequence of \eqref{infi}, we have the following useful estimations:
\begin{lemma}\label{esti}
For $|y|>2, \alpha>\beta>0$, it holds that
\begin{equation}
\int_{\mathbb{R}^d} Q^{\alpha}(x)  Q^{\beta}(x+y)dx \approx |y|^{-\frac{\beta(d-1)}{2}} e^{-\beta|y|}.
\end{equation}
\end{lemma}
\begin{proof}
On the one hand,
\begin{equation}
\begin{aligned}
\int_{\mathbb{R}^d} Q^{\alpha}(x)  Q^{\beta}(x+y)dx &\geq \int_{|x|<1} Q^{\alpha}(x)  Q^{\beta}(x+y)dx \\&\gtrsim |y|^{-\frac{\beta(d-1)}{2}} e^{-\beta|y|} .
\end{aligned}
\end{equation}
On the other hand,
since
\[(1+|x|)^{-\frac{(d-1)}{2}}  (1+|x+y|)^{-\frac{(d-1)}{2}} \lesssim  (1+|y|)^{-\frac{(d-1)}{2}}.\]
We have
\begin{equation}\label{es}
\begin{aligned}
   Q(x)  Q(x+y)
   & \lesssim (1+|x|)^{-\frac{(d-1)}{2}}  (1+|x+y|)^{-\frac{(d-1)}{2}}e^{- |x|-|x+y|} 
   \\&\lesssim (1+|y|)^{-\frac{(d-1)}{2}} e^{-|y|}.
   \end{aligned}
   \end{equation}
Thus,
\begin{equation}
\begin{aligned}
\int_{\mathbb{R}^d} Q^{\alpha}(x)  Q^{\beta}(x+y)dx &\lesssim
(1+|y|)^{-\frac{\beta(d-1)}{2}}e^{-\beta |y|}\int_{\mathbb{R}^d} Q^{\alpha-\beta}(x)dx \\&
\lesssim(1+|y|)^{-\frac{\beta(d-1)}{2}}e^{-\beta |y|}.
\end{aligned}
\end{equation}
\end{proof}

When $\alpha=\beta=2$, the weak interaction between two ground states depends on the dimension, which is the main reason why Theorem \ref{theorem1.2} depends on the dimension when $p=2$.
\begin{lemma}\label{lem2-2}
For $|y|>2$, we have
\begin{equation*}
\int\limits_{\mathbb{R}^d} Q(x)^2 Q(x+y )^2 dx  \approx
  \begin{cases}
   |y| e^{-2|y|}  ,&\text{if }  d=1;\\
    |y|^{-\frac{1}{2}} e^{-2|y|}  ,&\text{if }  d=2;\\
   |y|^{-2} e^{-2|y|} (\ln |y|) ,&\text{if }  d=3;\\
  |y|^{-(d-1)} e^{-2|y|} ,&\text{if } d>3.
  \end{cases}
  \end{equation*}
\end{lemma}
\begin{proof}
Without loss of generality,  set $R=|y|$ and assume $y=Re_1$, where $e_1=(1,0,\ldots,0)$.
For $d=1$, by direct calculation, we have
\begin{equation}\label{1111}
\int_{-\infty}^{+\infty} Q(x)^2 Q(x+R )^2 dx
\approx \int_{-\infty}^{+\infty} e^{-2|x|-2|x+R|} dx
  \approx R e^{-2R} .
\end{equation}

When $d\geq 2$ ,
by symmetry, we obtain
\begin{equation} \int_{\mathbb{R}^d } Q(x )^2 Q(x+Re_1 )^2 dx=2\int_{\{x_1>-\frac{R}{2} \}} Q(x )^2 Q(x+Re_1 )^2 dx.\end{equation}
In the region $\{x^1>-\frac{R}{2} \}\cap B_{\frac{2}{3}R}^c$, where $x^{1}$ denotes the first coordinate of $x$, we have $|x|\geq \frac{2R}{3}$ and $|x+R e_1|\geq \frac{R}{2}$, then
\[  e^{- |x|-|x+R e_1|} \leq   e^{-\frac{1}{7} |x| -\frac{6}{7}\cdot \frac{2R}{3}-\frac{R}{2} } \leq e^{-\frac{1}{7} |x|-\frac{15}{14}R}. \]
This implies that
\begin{equation}
\begin{aligned}
  \int_{\{x^1>-\frac{R}{2} \}\cap B_{\frac{2R}{3}}^c} & Q(x)^2 Q(x+Re_1 )^2 dx\\&
  \lesssim   \int_{\{x^1>-\frac{R}{2} \}\cap B_{\frac{2R}{3}}^c}   e^{- 2|x|-2|x+R e_1|} dx\\&
  \lesssim 
  e^{-\frac{15}{7}R} \int_{\mathbb{R}^d} e^{-\frac{2}{7}|x|}  dx
  =o_R( R^{-(d-1)} e^{-2R}).
  \end{aligned}
\end{equation}
While it is easy to see that
\[  \int_{B_{1} }  Q(x)^2 Q(x+Re_1 )^2 dx\approx R^{-(d-1)} e^{-2R}.\]
Now
using polar coordinates and observing that the  function $$ (1+|x|)^{-(d-1)} (1+|x+R e_1 |)^{-(d-1)} e^{-2 |x|-2|x+R e_1|}$$
depends only on the radial distance $r=|x|$ and the angle $\theta=\arccos\frac{x_1}{|x|}$, we obtain
\begin{equation}\label{4.151}
\begin{aligned}
\int_{B_{\frac{2R}{3}}\setminus B_{1} }&  Q(x )^2 Q(x+Re_1 )^2 dx \\
&\approx \int_{B_{\frac{2R}{3}}\setminus B_{1} } |x|^{-(d-1)} |x+R e_1 |^{-(d-1)} e^{-2 |x|-2|x+R e_1|} dx\\
&\approx \int_1^{\frac{2R}{3}} \int_0^\pi  (r^2 +2Rr\cos\theta +R^2 )^{-\frac{d-1}{2}} e^{-2r-2\sqrt{r^2+2Rr\cos\theta +R^2 }} (\sin\theta)^{d-2} d\theta dr\\
&\approx \int_1^{\frac{2R}{3}} \int_{\frac{\pi}{2}}^{\pi}  (r^2 +2Rr\cos\theta +R^2 )^{-\frac{d-1}{2}} e^{-2r-2\sqrt{r^2+2Rr\cos\theta +R^2 }} (\sin\theta)^{d-2} d\theta dr\\
&\approx \int_1^{\frac{2R}{3}} \int_{R-r}^{\sqrt{R^2+r^2}}  \frac{1}{Rr} (\sin\theta)^{d-3}t^{2-d} e^{-2r-2t} dtdr\\
&\approx \int_1^{\frac{2R}{3}} \frac{1}{R^{d-1}r}  e^{-2R} \int_{0}^{r+\sqrt{R^2+r^2}-R} (\sin\theta)^{d-3} e^{-2s}dsdr.
%&\approx   R^{-(d-1)} e^{-2R} \ln R,
\end{aligned}
\end{equation}
Here, we have made the substitution $t=\sqrt{r^2+2Rr\cos\theta +R^2 } $ and $s=t+r-R.$   Denote $$r'=\frac{r}{R} \in (0, \frac{2}{3}], \quad s'=\frac{s}{R} \in [0,\frac{3r'}{2}] \quad (\mbox{since } \sqrt{R^2+r^2}-R\leq \frac{r}{2}).$$
 By direct calculation,  we obtain:
\begin{equation}
  \begin{aligned}
  \sin^2 \theta&=1-\Big{(}\frac{(s'+1-r')^2-r'^2 -1 }{2r'}\Big{)}^2\\&
  =1-\Big{(}1-\frac{s'^2 +2s' (1-r')}{2r'}\Big{)}^2 
  \\&
  =\frac{s'}{r'}(s'+2-2r')[1-\frac{s'}{4r'}(s'+2-2r')]\\&
  \approx \frac{s'}{r'}  \quad ( \mbox{since }\frac{1}{4}\leq 1-\frac{s'}{4r'}(s'+2-2r')\leq 1 )
  \end{aligned}
\end{equation}
Thus, $\sin\theta\approx \sqrt{\frac{s}{r}}$ , and \eqref{4.151} yields
\begin{equation}\label{4.16}
\begin{aligned}
\int_{B_{\frac{2}{3}R}\setminus B_{1} } Q(x )^2 Q(x+Re_1 )^2 dx
&\approx R^{-(d-1)} e^{-2R} \int_1^{\frac{2}{3}R}  \int_{0}^{r+\sqrt{R^2+r^2}-R} r^{-1} (\frac{s}{r})^{\frac{d-3}{2}} e^{-2s}dsdr\\
&\approx  R^{-(d-1)} e^{-2R} \int_{1}^{\frac{2}{3}R} r^{-\frac{d-1}{2}} dr \int_0^{r+\sqrt{R^2+r^2}-R} e^{-2s} s^{\frac{d-3}{2}} ds\\
&\approx 
\begin{cases}
  R^{-\frac{1}{2}} e^{-2R}, & \mbox{if } d=2, \\
  R^{-2} e^{-2R}\ln R, & \mbox{if } d=3, \\
R^{-(d-1)} e^{-2R} , & \mbox{if } d>3.
\end{cases}
\end{aligned}
\end{equation}
Here, we have used the fact that $$\int_0^{r+\sqrt{R^2+r^2}-R} e^{-2s} s^{\frac{d-3}{2}} ds\approx1$$  since  $ e^{-2s} s^{\frac{d-3}{2}} \in L^1(0,+\infty)$.
Moreover, we note that   \begin{equation*}
\begin{aligned}
  \int_{B_{\frac{2}{3}R}\setminus B_{1} } Q(x )^2 Q(x+Re_1 )^2 dx &\leq   \int_{\mathbb{R}^d } Q(x )^2 Q(x+Re_1 )^2 dx\\& \leq
 2\int_{\{x_1>-\frac{R}{2} \}} Q(x )^2 Q(x+Re_1 )^2 dx\\&
 \leq 2\Big{(} \int_{B_{\frac{2}{3}R} } Q(x )^2 Q(x+Re_1 )^2 dx\\&  +   \int_{\{x^1>-\frac{R}{2} \}\cap B_{\frac{2R}{3}}^c}  Q(x)^2 Q(x+Re_1 )^2 dx\Big{)}.
  \end{aligned} 
\end{equation*}
Therefore, combining equations \eqref{1111} to \eqref{4.16}, we  complete the proof.
\end{proof}

Next, we present some estimates for the term $$ \|(Q + Q(\cdot + y))^p - Q^p - Q(\cdot + y)^p\|_{L^2}. $$ 
%which can be roughly regarded as the order of \( \|\rho\|_{H^1} \). 
The case 
$p=2$ is covered by Lemma \ref{lem2-2}. 
When $p>2$, it is easy to check that $$\|(Q + Q(\cdot + y))^p - Q^p - Q(\cdot + y)^p\|_{L^2} \approx  |y|^{-\frac{d-1}{2}}e^{-|y|}.$$
Since this result is not needed in the rest of the proof, we omit it.
We now consider the case $1<p<2$:

\begin{lemma}\label{lem1-4}
For $1<p<2$, and for $|y|$ large enough, we have
  \begin{equation}
\int_{\mathbb{R}^d}  |(Q(x)+Q(x+y))^p-Q(x)^p-Q(x+y)^p|^2 dx\approx |y|^{(\frac{1}{2}-p)(d-1)}e^{-p|y|}.
  \end{equation}
\end{lemma}
\begin{proof}
Without loss of generality, we assume $y=Re_1$. Then
\begin{equation}\label{3.111}
\begin{aligned}
(Q(x)+&Q(x+Re_1))^p-Q(x)^p-Q(x+Re_1)^p \\
&\approx
\begin{cases}
 Q(x)^{p-1}Q(x+Re_1), & \mbox{if } x^1>-\frac{R}{2} ,\\
 Q(x)Q(x+Re_1)^{p-1}, & \mbox{if } x^1 \le -\frac{R}{2}.
 \end{cases}
  \end{aligned}
\end{equation}
By the symmetry,
\begin{equation}
\begin{aligned}
\int_{\mathbb{R}^d} &|(Q(x)+Q(x+Re_1))^p-Q(x)^p-Q(x+Re_1)^p|^2 dx\\&=2\int_{\{x^1>-\frac{R}{2}\}}Q(x)^{2(p-1)}Q^2(x+Re_1) dx\\&
=2\int_{B_{\frac{R}{2}}}Q(x)^{2(p-1)}Q^2(x+Re_1) dx +2\int_{\{x^1>-\frac{R}{2}\}\cap B_\frac{R}{2}^c}Q(x)^{2(p-1)}Q^2(x+Re_1)dx .
\end{aligned}
\end{equation}
Denote $\alpha$ as the angle between the negative $x^1$-axis and the point $x$,  
so that $x^1=-|x|\cos \alpha$. 
 Set $$A:=\{x^1>-\frac{R}{2}\}\cap B_\frac{R}{2}^c \cap \{|\alpha| \le R^{-\frac{d}{d+1}} \}.$$  In region $A$, we have
 $$\text{dist}(x,  B_\frac{R}{2} \cap \{|\alpha| \le R^{-\frac{d}{d+1}}\})\leq \frac{R}{2\cos \alpha}-\frac{R}{2}\lesssim R\alpha^2 \lesssim R^{-\frac{d-1}{d+1}},$$ 
 Thus,
$$|A|\lesssim
\int_{\frac{R}{2}}^{\frac{R}{2}+CR^{-\frac{d-1}{d+1}}}   r^{d-1} dr \int_{0}^{R^{-\frac{d}{d+1}}} (\sin\alpha)^{d-2}  d\alpha\lesssim
 R^{\frac{d(d-1)}{d+1}}R^{-\frac{d(d-1)}{d+1}}  \lesssim 1.$$
It follows that
\begin{equation}
\begin{aligned}
\int_A Q(x)^{2(p-1)} Q^2(x+Re_1) dx &\lesssim \int_A R^{-p(d-1)}  e^{-pR} dx \\ &\lesssim  R^{-p(d-1)} e^{-pR} .
\end{aligned}
\end{equation}
Additionally, in the region $(\{x^1>-\frac{R}{2}\}\cap B_\frac{R}{2}^c) \setminus A$,  we have $$|x+Re_1 |\ge \sqrt{(\frac{R}{2})^2 +(\frac{R}{2}\sin R^{-\frac{d}{d+1}})^2}\geq  \frac{R}{2}+\frac{R^{\frac{1}{d+1}}}{8}.$$
Therefore, 
\begin{equation}
\begin{aligned}
\int_{(\{x^1>-\frac{R}{2}\}\cap B_\frac{R}{2}^c) \setminus A } & Q(x)^{2(p-1)} Q^2(x+Re_1) dx\\
 &\lesssim \int_{B_\frac{R}{2}^c } R^{-p(d-1)}  e^{-2|x|(p-1)} e^{-R-\frac{R^{\frac{1}{d+1}}}{4} } dx \\
&=o_R(R^{-p(d-1)} e^{-pR} ) .
\end{aligned}
\end{equation}

On the other hand,
\begin{equation}
\begin{aligned}
\int_{B_{\frac{R}{2}}}&Q(x)^{2(p-1)}Q^2(x+Re_1) dx\\
&=\int_{B_{1}}Q(x)^{2(p-1)}Q^2(x+Re_1) dx+\int_{B_{\frac{R}{2}}\setminus B_{1}}Q(x)^{2(p-1)}Q^2(x+Re_1) dx\\&
\approx R^{-(d-1)}e^{-2R}+\int_{B_{\frac{R}{2}}\setminus B_{1}}Q(x)^{2(p-1)}Q^2(x+Re_1) dx.
\end{aligned}
\end{equation}
And using the  same arguments in  \eqref{4.151}, we deduce that
\begin{equation}\label{3.166}
\begin{aligned}
&\int_{B_{\frac{R}{2}}\setminus B_{1}}Q(x)^{2(p-1)}Q^2(x+Re_1) dx\\&
\approx \int_{B_{\frac{R}{2}}\setminus B_{1}} |x|^{-(p-1)(d-1)} |x+R e_1 |^{-(d-1)} e^{-2(p-1) |x|-2|x+R e_1|} dx\\&
\approx\int_1^{\frac{R}{2}} \int_{\frac{\pi}{2}}^\pi r^{(2-p)(d-1)} (r^2 +2Rr\cos\theta +R^2 )^{-\frac{d-1}{2}} e^{-2(p-1)r-2\sqrt{r^2+2Rr\cos\theta +R^2 }} (\sin\theta)^{d-2} d\theta dr\\
&\approx \int_1^{\frac{R}{2}} \int_{R-r}^{\sqrt{R^2+r^2}}  \frac{r^{(2-p)(d-1)-1}}{R} (\sin\theta)^{d-3}t^{2-d} e^{-2(p-1)r-2t} dtdr\\
&\approx \int_1^{\frac{R}{2}} \frac{r^{(2-p)(d-1)-1}}{R^{d-1}}  e^{-2R} e^{2(2-p)r} \int_{0}^{r+\sqrt{R^2+r^2}-R} (\sin\theta)^{d-3}  e^{-2s}dsdr\\
%&\approx \int_1^{\frac{R}{2}} \frac{r^{(2-p)(d-1)-1}}{R^{d-1}}  e^{-2R} e^{2(2-p)r} \int_{0}^{2r} \sqrt{\frac{s}{r}}^{d-3}  e^{-2s}dsdr\\
&\approx \int_1^{\frac{R}{2}} \frac{r^{(2-p)(d-1)-1-\frac{d-3}{2}}}{R^{d-1}}  e^{-2R} e^{2(2-p)r} dr
\\
&\approx R^{(\frac{1}{2}-p)(d-1)}e^{-pR}.
\end{aligned}
\end{equation}
Combine \eqref{3.111}-\eqref{3.166} we complete the proof.
\end{proof}

%As for the case $p>2$, the following estimates can be easily obtained:
%\begin{lemma}\label{Lemma 3.4}
%  For $p>2$ and  $|y|>1$, we have  \begin{equation}
% \int_{\mathbb{R}^d} |(Q(x)+Q(x+y))^p-Q(x)^p-Q(x+y)^p|^2 dx\approx |y|^{-(d-1)}e^{-2|y|}.
%  \end{equation}
%\end{lemma}
%\begin{proof}
 % Since $$(Q(x)+Q(x+y))^p-Q(x)^p-Q(x+y)^p\lesssim  Q^{p-1}(x)Q(x+y)+Q^{p-1}(x+y)Q(x),$$ we obtain 
 % \[ \int_{\mathbb{R}^d}  |(Q(x)+Q(x+y))^p-Q(x)^p-Q(x+y)^p|^2 dx\lesssim |y|^{-(d-1)}e^{-2|y|}.\]
 % On the other hand, since 
 % \begin{equation*}
 % \begin{aligned}
 %   \int_{\mathbb{R}^d}  |(Q(x)+Q(x+y))^p-&Q(x)^p-Q(x+y)^p|^2 dx\\&\geq \int_{B_{1}}|(Q(x)+Q(x+y))^p-Q(x)^p-Q(x+y)^p|^2dx\\&\gtrsim |y|^{-(d-1)}e^{-2|y|}.
 %   \end{aligned}
 % \end{equation*}
  %  This completes the proof.
%\end{proof}

Finally, we provide estimates for the interaction integrals between the gradient of a ground state and another ground state.

\begin{lemma}\label{lemmay}
For $|y|>1, p>1$, there exists a constant $\bar{c}>0$ such that
\begin{equation}
\Big{|}\int_{\mathbb{R}^d} Q^{p-1}(x)  \nabla Q(x) Q(x+y)dx-\bar{c}\frac{y}{|y|}|y|^{-\frac{d-1}{2}}e^{-|y|} \Big{|} \lesssim |y|^{-\frac{d+1}{2}}e^{-|y|}.
\end{equation}
\end{lemma}
\begin{proof}
Choose  $\theta\in (0, 1)$ such that $p\theta>1$. Then by \eqref{infi}, we have
\begin{equation}\label{2.9}
\begin{aligned}
 \Big{|}\int_{|x|>\theta |y|} Q^{p-1}(x)  \nabla Q(x) Q(x+y)dx\Big{|}
& \leq C|y|^{\frac{p(1-d)}{2}}e^{-p\theta |y|} \int_{|x|>\theta |y|} Q(x+y)dx\\& \leq C|y|^{-\frac{d+1}{2}}e^{-|y|}
 .
 \end{aligned}
 \end{equation}
 On the other hand, in the region $|x|\leq \theta |y|$,
 we have
 \[ \Big{|}|x+y|^{-\frac{d-1}{2}}- |y|^{-\frac{d-1}{2}}\Big{|}\lesssim {|x|}{|y|^{-\frac{d+1}{2}}} ,  \]
% and
 % \[  |\frac{x+y}{|x+y|}-\frac{y}{|y|}|\lesssim  \frac{|x|}{|y|},  \]
 and
 \[ \Big{|}|x+y|-|y|-\frac{y\cdot x}{|y|}\Big{|}\lesssim  \frac{|x|^2}{|y|}.\]
 It follows from the mean value theorem that
   \begin{equation}\label{2.10}
 \begin{aligned}
&\Big{|}|x+y|^{-\frac{d-1}{2}} e^{-|x+y|}-|y|^{-\frac{d-1}{2}} e^{-|y|-\frac{y\cdot x}{|y|}} \Big{|}
\\&\leq
\Big{|} |x+y|^{-\frac{d-1}{2}} e^{-|x+y|}   -|y|^{-\frac{d-1}{2}} e^{-|x+y|}   \Big{|}+ \Big{|}|y|^{-\frac{d-1}{2}} e^{-|x+y|}- |y|^{-\frac{d-1}{2}} e^{-|y|-\frac{y\cdot x}{|y|}}   \Big{|}\\&
\lesssim \frac{|x|}{|y|}|y|^{-\frac{d-1}{2}} e^{-|x+y|}
+ \frac{|x|^2}{|y|}|y|^{-\frac{d-1}{2}} e^{-|y|+|x|} \\&
\lesssim (1+|x|^2)|y|^{-\frac{d+1}{2}} e^{-|y|+|x|}
.
 \end{aligned}
 \end{equation}
 % \begin{equation}
% \begin{aligned}
%&|\frac{x+y}{|x+y|}|x+y|^{-\frac{d-1}{2}} e^{-|x+y|}  - \frac{y}{|y|} |y|^{-\frac{d-1}{2}} e^{-|y|-\frac{y\cdot x}{|y|}} |
%\leq |\frac{y}{|y|}|y|^{-\frac{d-1}{2}} e^{-|x+y|}- \frac{y}{|y|} |y|^{-\frac{d-1}{2}} e^{-|y|-\frac{y\cdot x}{|y|}}   |\\&
%+|\frac{x+y}{|x+y|} |x+y|^{-\frac{d-1}{2}} e^{-|x+y|} - \frac{y}{|y|}|x+y|^{-\frac{d-1}{2}} e^{-|x+y|}|
%+|  \frac{y}{|y|}|x+y|^{-\frac{d-1}{2}} e^{-|x+y|}   -\frac{y}{|y|}|y|^{-\frac{d-1}{2}} e^{-|x+y|}   |\\&
%\lesssim \frac{|x|}{|y|}|y|^{-\frac{d-1}{2}} e^{-|y|+|x|}
%+ \frac{|x|^2}{|y|}|y|^{-\frac{d-1}{2}} e^{-|y|+|x|} \\&
%\lesssim \frac{|x|^2}{|y|}|y|^{-\frac{d-1}{2}} e^{-|y|+|x|}
%.
% \end{aligned}
% \end{equation}
By \eqref{infi}, we deduce that
 \begin{equation}\label{2.11}
 \begin{aligned}
 \Big{|}\int_{|x|\leq \theta |y|} (Q(x+y)- &c_{Q} |x+y|^{-\frac{d-1}{2}} e^{-|x+y|}) \nabla Q(x)Q^{p-1}(x) dx
\Big{|}
 \\&\lesssim  \int_{|x|\leq \theta |y|}  |x+y|^{-\frac{d+1}{2}}(1+|x|)^{-\frac{d-1}{2}} e^{-|x+y|-|x|}   Q^{p-1}(x)   dx\\&
 \lesssim  |y|^{-\frac{d+1}{2}} e^{-|y|}.
 \end{aligned}
 \end{equation}
 Additionally, we deduce from \eqref{2.10} and \eqref{infi} that
  \begin{equation}
 \begin{aligned}
 \Big{|} \int_{|x|\leq\theta |y|} &( |x+y|^{-\frac{d-1}{2}} e^{-|x+y|}  -  |y|^{-\frac{d-1}{2}} e^{-|y|-\frac{y\cdot x}{|y|}}  )   Q^{p-1}(x)  \nabla Q(x) dx \Big{|}\\
&\lesssim\int_{|x|\leq\theta |y|} (1+|x|^2)|y|^{-\frac{d+1}{2}} e^{-|y|+|x|} (1+|x|)^{-\frac{p(d-1)}{2}}e^{-p|x|}dx\\&
\lesssim|y|^{-\frac{d+1}{2}}e^{-|y|}.
   \end{aligned}
  \end{equation}
 On the other hand,
   \begin{equation}\label{2.13}
 \begin{aligned}
\Big{|} \int_{|x|> \theta |y|}  &|y|^{-\frac{d-1}{2}} e^{-|y|-\frac{y\cdot x}{|y|}}     Q^{p-1}(x)  \nabla Q(x) dx\Big{|}\\&
 \lesssim|y|^{-\frac{d-1}{2}} e^{-|y|} \int_{|x|> \theta |y|} |x|^{-\frac{p(d-1)}{2}}e^{-(p-1)|x|}dx\\&
 \lesssim |y|^{-\frac{d+1}{2}}e^{-|y|}.
  \end{aligned}
 \end{equation}
 Consequently, we conclude from \eqref{2.9}, \eqref{2.11}-\eqref{2.13} that
 \begin{equation}
 \begin{aligned}
\Big{|}\int_{\mathbb{R}^d} & Q^{p-1}(x)  \nabla Q(x)( Q(x+y) -c_{Q} |y|^{-\frac{d-1}{2}} e^{-|y|-\frac{y\cdot x}{|y|}}) dx \Big{|}  \\
&\lesssim |y|^{-\frac{d+1}{2}}e^{-|y|}.
\end{aligned}
\end{equation}
By integrating by parts, we obtain
$$ p\int_{\mathbb{R}^d} e^{-\frac{y\cdot x}{|y|}}     Q^{p-1}(x)  \nabla Q(x) dx
=\frac{y}{|y|}\int_{\mathbb{R}^d}  e^{-\frac{y\cdot x}{|y|}}     Q^{p}(x) dx
=\frac{y}{|y|}\int_{\mathbb{R}^d}  e^{-x^1}     Q^{p}(x) dx.
     $$
     Thus,    choosing $$\bar{c}=\frac{c_Q}{p}\int_{\mathbb{R}^d}  e^{-x^1}     Q^{p}(x) dx,$$   we complete the proof.
\end{proof}

\subsection{The complete proof of Theorem \ref{theorem1.2}}
Before presenting the proof of Theorem \ref{theorem1.2}, we first provide a detailed analysis of the two main estimates as mentioned in Section 1.2.
 
For a set of points \( \{y_i\}_{i=1}^{m} \subset \mathbb{R}^d \), recall the notations:
\[
\sigma = \sum_{i=1}^{m} Q_i, \quad Q_i(x) := Q(x + y_i).
\]
Let
\[
\mathscr{F} := \text{span}\left\{ \partial_j Q_i \mid 1 \leq j \leq d, 1 \leq i \leq m \right\}.
\]
Denote by \( P_{\mathscr{F}^\perp} \) (respectively \( P_{\mathscr{F}} \)) the projection from \( H^1 \) onto \( \mathscr{F}^\perp \) (respectively \( \mathscr{F} \)).
Define the compact operator \( K \) by
\[
K = p(-\Delta + 1)^{-1} \sigma^{p-1}: H^1 \to H^1.
\]

We have the following result:
\begin{lemma}\label{reduction}
There exists $\widetilde{R}>0$  such that, 
if $R:=\min\limits_{i\neq j}|y_i-y_j|\geq \widetilde{R}$, then the problem
$$  (\mathbf{id}-P_{\mathscr{F}^\perp}K) v=\varphi $$
 is uniquely solvable for all  $\varphi \in \mathscr{F}^\perp$, and the following estimate holds:
\begin{equation}
\|v\|_{H^1}\approx \|\varphi\|_{H^1}.
\end{equation}
\end{lemma}
\begin{proof}
  We first prove the a priori estimate 
  \begin{equation}\label{prio}\|v\|_{H^1}\leq C\|\varphi\|_{H^1}.\end{equation}
   Otherwise, suppose there exist a sequence $\varphi_n$ such that $\|\varphi_n\|_{H^1}=o_n(1)$ in $\mathscr{F}_n^\perp$,  along with $\{y_{i,n}\}$ such that $\min\limits_{i\neq j}|y_{i,n}-y_{j,n}|\to +\infty$ and a sequence 
  $v_n\in \mathscr{F}_n^\perp$ such that
  $$(\mathbf{id}-P_{\mathscr{F}_n^\perp} K_n) v_n=\varphi_n \in \mathscr{F}_n^\perp, \|v_n\|_{H^1}=1.$$
Rewriting this, we obtain
  \[(-\Delta+1) v_n- p(-\Delta+1) P_{\mathscr{F}_n^\perp} (-\Delta+1)^{-1}\sigma_n^{p-1} v_n = (-\Delta+1) \varphi_n.  \]
  Therefore,
  \begin{equation}
  (-\Delta+1) v_n- p\sigma_n^{p-1} v_n =   (-\Delta+1) \varphi_n -p \sum\limits_{i,j}\int_{\mathbb{R}^d} \sigma_n^{p-1} v_n e_{ij}^{(n)} dx  (-\Delta+1)e_{ij}^{(n)}.
  \end{equation}
 Here, \( e_{ij}^{(n)} \) denotes an orthonormal basis of \( \mathscr{F}_n \). In fact, the set \( \{ e_{ij}^{(n)} \} \) can be obtained by applying the Gram-Schmidt orthogonalization process to \( \{ \partial_j Q_i^{(n)} \}_{ 1 \leq j \leq d, 1 \leq i \leq m} \). Therefore, for fixed $i_0,j_0$, we may assume that
$$
e_{i_0 j_0}^{(n)} - \frac{\partial_{j_0} Q_{i_0}^{(n)}}{\|\partial_{j_0} Q_{i_0}^{(n)}\|_{H^1}} = \sum_{i , j} o_n(1) \partial_j Q_i^{(n)}.
$$
Since \( v_n \in \mathscr{F}_n^\perp \), we have
  \[ |\int_{\mathbb{R}^d} \sigma_n^{p-1} v_n e_{i_0j_0}^{(n)} dx|\lesssim  \sum\limits_{i} \Big{|}\int_{\mathbb{R}^d} (\sigma_n^{p-1}-|Q_i^{(n)}|^{p-1}) v_n  \nabla Q_i^{(n)} dx\Big{|} =o_n(1).\]
Define \( v_{n,i} := v_n(\cdot - y_{n,i}) \). Up to a subsequence, \( v_{n,i} \) converges weakly in \( H^1 \) to a function \( v_i \), which satisfies the equation
\[
(-\Delta + 1) v_i - p Q^{p-1} v_i = 0, \quad \int_{\mathbb{R}^d} Q^{p-1} \nabla Q v_i dx = 0.
\]
Therefore, \( v_i = 0 \), and consequently, \( v_{n,i} \to 0 \) in \( L^2_{\text{loc}} \).

For a fixed \( T \) large enough such that $Q^{p-1}(x)\leq \frac{1}{100pm^p}$ in $B_T^c$, test the equation with \( v_n \), we obtain
  \begin{equation}
  \begin{aligned}
\|v_n\|_{H^1}^2&= \int_{\mathbb{R}^d} p\sigma_n^{p-1} v_n^2 dx +o_n(1)\\&
\leq C\int_{\cup_{i=1}^{m} B_{T}(-y_{n,i})}  v_n^2 dx + \int_{\mathbb{R}^d\setminus \cup_{i=1}^{m} B_{T}(-y_{n,i})} p\sigma_n^{p-1} v_n^2 dx +o_n(1)
\\&
\leq o_n(1)+\frac{1}{100}\|v_n\|_{L^2}^2\\&
\leq \frac{1}{2}.
\end{aligned}
\end{equation}
This leads to a contradiction with $\|v_n\|_{H^1}=1.$ Therefore, we obtain \eqref{prio} and  the  solvability can be obtained by Fredholm's alternative.

On the other hand,
 since $\|K\|_{\mathscr{L}(H^1, H^1)}\leq pm^p,$ we conclude that
$\|\varphi\|_{H^1}\leq C \|v\|_{H^1}$ uniformly with respect to $R$.
\end{proof}

Following the above notations, we obtain the following result.
 
\begin{proposition}\label{p1}
Let  $u=\sigma+\rho$ , and  $R:=\min\limits_{i\neq j}|y_i-y_j|\geq \widetilde{R}$. Then there exists $\delta>0$ (independent of $R$), such that,  if $\rho \in \mathscr{F}^{\perp} $  and $\|\rho\|_{H^1}<\delta$, we have the estimate
\begin{equation}\label{third}
  \|\rho\|_{H^1}
  \lesssim
    \|f\|_{L^2}  +\|h\|_{H^{-1}},
\end{equation}
where $f$ and $h$  are defined in \eqref{fh}.
\end{proposition}
\begin{proof}
Recall that, by direct calculation, we have
\begin{equation}\label{3-13}
\begin{aligned}
(-\Delta+1) \rho-p\sigma^{p-1} \rho
&= f+h +  N(\rho),
\end{aligned}
\end{equation}
where \begin{equation}
f:=(\sum\limits_{i=1}^{m}  Q_i )^p -\sum\limits_{i=1}^{m}  Q_i^p,\quad h:=-\Delta u+u -|u|^{p-1}u
\end{equation}
and
  \begin{align*}
  N(\rho):=|\sigma+\rho|^{p-1}(\sigma+\rho) -\sigma^p -p\sigma^{p-1} \rho=
  \begin{cases}
    O(\rho^2), & \mbox{if } p>2 \\
    O(\rho^p), & \mbox{if } 1<p\leq2.
  \end{cases}
  \end{align*}
   By Lemma \ref{reduction},  since $\rho$ is a solution to
   \begin{equation*}
\rho-P_{\mathscr{F}^\perp}K \rho
= P_{\mathscr{F}^\perp} (-\Delta+1)^{-1}(f+h+N(\rho)),
\end{equation*}
     we have
  \begin{equation}
  \begin{aligned}
  \|\rho\|_{H^1}&\lesssim \|P_{\mathscr{F}^\perp} (-\Delta+1)^{-1}(f+h+N(\rho)) \|_{H^1}\\&
  \lesssim \|f+h+N(\rho)\|_{H^{-1}}\\&
  \lesssim \begin{cases}
             \|f\|_{L^2}  +\|h\|_{H^{-1}}+\|\rho\|_{H^1}^{2} , & \mbox{if } p>2 \\
             \|f\|_{L^2}  +\|h\|_{H^{-1}}+\|\rho\|_{H^1}^{p} , & \mbox{if }  1<p\leq2.
           \end{cases}
  \end{aligned}
  \end{equation}
\end{proof}

 In order to compare  $\|\rho\|_{H^1}, \|f\|_{L^2} $ and $\|h\|_{H^{-1}}$, we still need: 
\begin{proposition}\label{prop3.} Under the same assumption in Proposition  \ref{p1},
%for $R:=\min\limits_{i\neq j}|y_i-y_j|$ large enough, 
we have
  \begin{equation}\label{second}
  R^{-\frac{d-1}{2}}e^{-R}\lesssim \|h\|_{H^{-1}}+\|\rho\|_{H^1}^{2}+ e^{-(1-\min\{1,\frac{p(p-1)}{4} \}) R}\|\rho\|_{H^1}^{\min\{p-1,1\}}.
  \end{equation}
\end{proposition}
\begin{proof}
  Test the equation \eqref{3-13} with $\nabla Q_{i_0}$ for some fixed $i_0$, we obtain
  \begin{equation}
  \begin{aligned}
    \int_{\mathbb{R}^d} (f+h+N(\rho))\nabla Q_{i_0} dx&=\int_{\mathbb{R}^d} -p\sigma^{p-1} \rho \nabla Q_{i_0} dx\\&=-p \int_{\mathbb{R}^d} (\sigma^{p-1}-Q_{i_0}^{p-1}) \rho \nabla Q_{i_0} dx.
    \end{aligned}
  \end{equation}
  Since 
  \begin{equation}
  \begin{aligned}
   (\sigma^{p-1}-Q_{i_0}^{p-1}) |\rho \nabla Q_{i_0}|& \lesssim 
  \begin{cases}
 |\rho| Q_{i_0}^{p-1}\sum \limits_{i\neq i_0} Q_i, & \mbox{if } Q_{i_0}>\sum\limits_{i\neq i_0} Q_i \\
   (\sum\limits_{i\neq i_0} Q_i)^{p-1} |\rho| Q_{i_0}, & \mbox{if } Q_{i_0}\leq \sum\limits_{i\neq i_0} Q_i 
  \end{cases}\\&
  \lesssim
    \begin{cases}
  \sum\limits_{i\neq i_0} Q_iQ_{i_0}|\rho|^{\min\{p-1,1\}}  , & \mbox{if } |\rho|< Q_{i_0}, Q_{i_0}>\sum\limits_{i\neq i_0} Q_i  \\
    |\rho|^{p+1}, & \mbox{if } |\rho|\geq Q_{i_0}>\sum\limits_{i\neq i_0} Q_i\\
     |\rho|^{p+1}, & \mbox{if } Q_{i_0}\leq \sum\limits_{i\neq i_0} Q_i\leq |\rho| \\
    \sum\limits_{i\neq i_0} Q_i |\rho|^{\min\{p-1,1\}} Q_{i_0}, & \mbox{otherwise}  ,
  \end{cases}
  \end{aligned}
  \end{equation}
 here, we have used the fact that $Q^{\max\{p-1,1\}}\lesssim Q$ since $Q$ is bounded. 
Thus, for small  $\varepsilon>0$,   via \eqref{es} and H\"{o}lder inequality,  we obtain
\begin{equation}\label{4.111}
\begin{aligned}
\Big{|}\int_{\mathbb{R}^d} (\sigma^{p-1}-Q_{i_0}^{p-1}) \rho \nabla Q_{i_0} dx\Big{|}&\lesssim \|\rho\|_{H^1}^{p+1}
+ \sum\limits_{i\neq i_0}  \int_{\mathbb{R}^d}   Q_i |\rho|^{\min\{p-1,1\}} Q_{i_0} dx \\&\lesssim
\|\rho\|_{H^1}^{p+1}
+ R^{-\frac{(1-\varepsilon)(d-1)}{2}} e^{-(1-\varepsilon) R}  \int_{\mathbb{R}^d}  |\rho|^{\min\{p-1,1\}} Q_{i_0}^{\epsilon} dx\\&
\lesssim
 \|\rho\|_{H^1}^{p+1}+ R^{-\frac{(1-\varepsilon)(d-1)}{2}} e^{-(1-\varepsilon) R}\|\rho\|_{H^1}^{\min\{p-1,1\}}.
 \end{aligned}
\end{equation}
In particular,
\begin{equation}\label{3.35}
\Big{|}\int_{\mathbb{R}^d} (\sigma^{p-1}-Q_{i_0}^{p-1}) \rho \nabla Q_{i_0} dx\Big{|}\lesssim \|\rho\|_{H^1}^{p+1}+ e^{-(1-\min\{1,\frac{p(p-1)}{4} \}) R}\|\rho\|_{H^1}^{\min\{p-1,1\}}.
\end{equation}
Next,
   \begin{equation}\label{3.36}
   \begin{aligned}
    |\int_{\mathbb{R}^d}  N(\rho)\nabla Q_{i_0} dx|&
    \lesssim   \int_{\mathbb{R}^d} \Big{|}|\sigma+\rho|^{p-1}(\sigma+\rho) -\sigma^p -p\sigma^{p-1} \rho\Big{|}Q_{i_0} dx\\&
    \lesssim \int_{|\rho|>\sigma} |\rho|^{p+1} dx + \int_{|\rho|\leq \sigma} \sigma^{p-1} |\rho|^2 dx\\&
    \lesssim  \|\rho\|_{H^1}^2.
    \end{aligned}
    \end{equation}
  Therefore, we obtain that
  \begin{equation}\label{3.16}
 | \int_{\mathbb{R}^d} f \nabla Q_{i_0} dx|\lesssim \|h\|_{H^{-1}}+\|\rho\|_{H^1}^{2}+ e^{-(1-\min\{1,\frac{p(p-1)}{4} \}) R}\|\rho\|_{H^1}^{\min\{p-1,1\}}.
  \end{equation}
 
Using Lemma \ref{point}, without loss of generality, we assume  that $$y_1=0, ~~y_{k}^{1}>0  ~~ \text{  and  }~~ \frac{y_{k}^{1}}{|y_{k}|}>c, ~~\forall k>1.$$
We deduce from \eqref{3.16} that
$$| \int_{\mathbb{R}^d} f \partial_1 Q_{1} dx|\lesssim \|h\|_{H^{-1}}+\|\rho\|_{H^1}^{2}+ e^{-(1-\min\{1,\frac{p(p-1)}{4} \}) R}\|\rho\|_{H^1}^{\min\{p-1,1\}}.$$
Denote $x=(x^1, x')$, then
\begin{equation}
\begin{aligned}
\int_{\mathbb{R}^d} f \partial_1 Q_{1} dx
=\int_{x^1<0} (f(x^1, x')-f(-x^1,x'))\partial_1 Q_{1} dx.
\end{aligned}
\end{equation}
Set \begin{equation}\label{R0}R_{i_0}:=\min_{j\neq i_0}|y_{i_0}-y_j|.\end{equation}
In the region $\{x^1<0\}$, we have $\partial_1 Q_{1}\geq 0$ and since $y_{k}^{1}>0, k\geq 2$, we have $f(x^1, x')-f(-x^1,x')>0$. Therefore,
\begin{equation}\label{first}
\begin{aligned}
|\int_{\mathbb{R}^d} f \nabla Q_{1} dx|&\geq 
\int_{\mathbb{R}^d} f \partial_1 Q_{1} dx\\&>\int_{\{-2<x^1<-1, |x'|<1\}} (f(x^1, x')-f(-x^1,x'))\partial_1 Q_{1} dx\\&
\gtrsim \int_{\{-2<x^1<-1, |x'|<1\}} f(x^1, x')-f(-x^1,x') dx \\&
\approx  R_1^{-\frac{d-1}{2}}e^{-R_1}.
\end{aligned}
\end{equation}
We have used the mean value theorem and in the region \( \{-2 < x^1 < 2, |x'| < 1 \} \),
\[
\partial_1 f \approx Q_1^{p-1} \sum_{j > 1} \partial_1 Q_j+Q_1^{p-2} \partial_1 Q_1\sum_{j > 1}  Q_j \approx \sum_{j > 1} |y_j|^{-\frac{d-1}{2}} e^{-|y_j|} \approx R_1^{-\frac{d-1}{2}} e^{-R_1}.
\]
due the fact that  \( y_k^1 \geq c|y_k|\geq  cR_1 \) for all \( k > 2 \) and \eqref{infi}.

Thus, we obtain the estimate
\[
R_1^{-\frac{d-1}{2}} e^{-R_1} \lesssim \|h\|_{H^{-1}}+\|\rho\|_{H^1}^{2}+ e^{-(1-\min\{1,\frac{p(p-1)}{4} \}) R}\|\rho\|_{H^1}^{\min\{p-1,1\}}.
\]

Next, 
if $R_2\geq R_1$, then 
\[R_2^{-\frac{d-1}{2}} e^{-R_2} \leq
R_1^{-\frac{d-1}{2}} e^{-R_1} \lesssim \|h\|_{H^{-1}}+\|\rho\|_{H^1}^{2}+ e^{-(1-\min\{1,\frac{p(p-1)}{4} \}) R}\|\rho\|_{H^1}^{\min\{p-1,1\}}.
\]
If $R_2<R_1$, then 
$R_{2}:=\min_{j>2}|y_{2}-y_j|$.
Denote $\widetilde{f}:= (\sum\limits_{i=2}^m  Q_i )^p -\sum\limits_{i=2}^{m}  Q_i^p$.
Note that
\begin{equation}
\begin{aligned}
|f-(\sum_{i=2}^m  Q_i )^p +\sum_{i=2}^m  Q_i^p |&=|(\sum_{i=1}^m  Q_i )^p - Q_1^p
-(\sum_{i=2}^m  Q_i )^p |\\&
\lesssim
\begin{cases}
  Q_1^p, & \mbox{if } Q_1>\sum\limits_{i=2}^m  Q_i \\
  \sum\limits_{i=2}^m  Q_i^{p-1} Q_1, & \mbox{otherwise}.
\end{cases}
\end{aligned}
\end{equation}
Thus
\begin{equation}
\begin{aligned}
|\int_{\mathbb{R}^d}  f \nabla Q_{2} dx
-\int_{\mathbb{R}^d} \widetilde{f} \nabla Q_{2} dx|&\lesssim  \int_{\mathbb{R}^d}  ( Q_1^p Q_2+   \sum_{i=2}^m  Q_i^{p-1} Q_1 Q_2) dx\\&
\lesssim   R_1^{-\frac{d-1}{2}}e^{-R_1}\\&
\lesssim \|h\|_{H^{-1}}+\|\rho\|_{H^1}^{2}+ e^{-(1-\min\{1,\frac{p(p-1)}{4} \}) R}\|\rho\|_{H^1}^{\min\{p-1,1\}}.
\end{aligned}
\end{equation}
Similar to the proof of \eqref{first} we obtain that
\[|\int_{\mathbb{R}^d}  \widetilde{f}  \nabla Q_{2} dx|\gtrsim R_2^{-\frac{d-1}{2}}e^{-R_2}, \]
therefore, in any case we obtain
\begin{equation}
  R_2^{-\frac{d-1}{2}}e^{-R_2}\lesssim \|h\|_{H^{-1}}+\|\rho\|_{H^1}^{2}+ e^{-(1-\min\{1,\frac{p(p-1)}{4} \}) R}\|\rho\|_{H^1}^{\min\{p-1,1\}},
  \end{equation}
and we finish the proof by induction.
\end{proof}

Now we can complete the proof of Theorem \ref{theorem1.2}.

\begin{proof}[Proof of Theorem \ref{theorem1.2}]
By equations \eqref{second} and \eqref{third}, we obtain
\[
R^{-\frac{d-1}{2}} e^{-R} \lesssim \|h\|_{H^{-1}} + \|f\|_{L^2}^{2} + e^{-(1-\min\{1,\frac{p(p-1)}{4} \}) R} (\|h\|_{H^{-1}}+\|f\|_{L^2})^{\min\{p-1,1\}} .
\]
For $p>2$, using Lemma \ref{lem3.1} , Lemma \ref{esti} and  \eqref{es}, one can deduce that 
\[
\|f\|_{L^2}\approx  R^{-\frac{d-1}{2}} e^{-R}.
\]
For $1<p\leq 2$,  via  Lemma \ref{lem3.1} we obtain
\[
\left( \sum_{i=1}^m Q_i \right)^p - \sum_{i=1}^m Q_i^p \leq \sum\limits_{\substack{i,j \\ i\neq j}} \left[ (Q_i + Q_j)^p - Q_i^p - Q_j^p \right].
\]
For any fixed $i_0, j_0$, it is easy to see that 
\[
\left( \sum_{i=1}^m Q_i \right)^p - \sum_{i=1}^m Q_i^p \geq    (Q_{i_0} + Q_{j_0})^p - Q_{i_0}^p - Q_{j_0}^p.
\]
Therefore, by Lemmas \ref{lem2-2} and \ref{lem1-4}, we obtain
\[
\|f\|_{L^2} \approx F_{d,p} \left( R^{-\frac{d-1}{2}} e^{-R} \right).
\]
In any cases, we have
  \[  \|f\|_{L^2}^{2} + e^{-(1-\min\{1,\frac{p(p-1)}{4} \}) R} \|f\|_{L^2}^{\min\{p-1,1\}}=o_R(R^{-\frac{d-1}{2}}e^{-R}). \]
 Therefore, $R^{-\frac{d-1}{2}}e^{-R}\lesssim \|h\|_{H^{-1}}$ and by \eqref{third},
  \begin{equation}
  \begin{aligned}
  \|\rho\|_{H^1}
  &\lesssim
    \|f\|_{L^2}  +\|h\|_{H^{-1}}\\&
    \lesssim F_{d,p}( R^{-\frac{d-1}{2}}e^{-R})+\|h\|_{H^{-1}}\\&
     \lesssim F_{d,p}( \|h\|_{H^{-1}})
    .
    \end{aligned}
\end{equation}
\end{proof}

\subsection{Sharp examples}
In this section, we always assume that  \( 1\leq p < 2 \) 
or \( p = 2, d = 1, 2, 3 \). 
To demonstrate the sharpness, we observe that by Lemma \ref{reduction},
%giving
% $$\sigma=\sum_{i=1}^{m}Q_i, Q_i(x)=Q(x+y_i), f=(\sum\limits_{i=1}^{m}  Q_i )^p -\sum\limits_{i=1}^{m}  Q_i^p,$$
% $$
 %N(\rho)=(|\sigma+\rho|^{p-1}(\sigma+\rho)- \sigma^p- p\sigma^{p-1}\rho),
% $$ 
for $f$ and $N(\rho)$ giving in \eqref{fh} and \eqref{N},
 we have:
 \begin{proposition}\label{rh}
   For $R:=\min\limits_{i\neq j}|y_i-y_j|$ large enough,
 we can find a solution $\rho\in \mathscr{F}^\perp$ to equation:
\begin{equation}\label{3.45}
\begin{aligned}
  (-\Delta+1) \rho- p\sigma^{p-1}\rho& = (-\Delta+1) P_{\mathscr{F}^\perp} (-\Delta+1)^{-1} (f+N(\rho)) \\&+ p \sum\limits_{i,j}\int_{\mathbb{R}^d}  \sigma^{p-1} \rho e_{ij} dx  (-\Delta+1) e_{ij}.
  \end{aligned}
  \end{equation}
  Here, $e_{ij}$ are orthonormal basis of $\mathscr{F}$.  In particular, we have
  \begin{equation}
  \begin{aligned}
   \|\rho\|_{H^1}^2&\approx \|P_{\mathscr{F}^\perp} (-\Delta+1)^{-1} f\|_{H^{1}}^2\\&
=\| f\|_{H^{-1}}^2-\|P_{\mathscr{F}} (-\Delta+1)^{-1} f\|_{H^{1}}^2
   .
   \end{aligned}
  \end{equation}
 \end{proposition}
\begin{proof}
\eqref{3.45} is  equivalent to solve
\[ \rho =  \textbf{T}(\rho):= (\mathbf{id}-P_{\mathscr{F}^\perp}K)^{-1} P_{\mathscr{F}^\perp} (-\Delta+1)^{-1} (f+N(\rho)). \]
For $R$ large enough, it is easy to check that
$$\textbf{T}(\cdot): \{\rho\in \mathscr{F}^\perp \mid \|\rho\|_{H^1}\leq \frac{1}{R}\} \to  \{\rho\in \mathscr{F}^\perp \mid \|\rho\|_{H^1}\leq \frac{1}{R}\} .$$
For $\rho_1, \rho_2\in \{\rho\in \mathscr{F}^\perp \mid \|\rho\|_{H^1}\leq \frac{1}{R}\}$,  we have
\begin{equation}
\begin{aligned}
 \| \textbf{T}(\rho_1)-\textbf{T}(\rho_2) \|_{H^1}^2&\lesssim  \|   N(\rho_1)-N(\rho_2)\|_{H^{-1}}^2\\&
 \lesssim
 \|   N(\rho_1)-N(\rho_2)\|_{L^2}^2\\&
 \lesssim  \int_{\mathbb{R}^d} [ (\sigma+|\rho_1|+|\rho_2|)^{p-1}-\sigma^{p-1}]^2|\rho_1-\rho_2|^2 dx \\&\lesssim
 o_{R}(1)\|\rho_1-\rho_2\|_{H^1}^2.
 \end{aligned}
\end{equation}
By applying the contraction mapping principle, we can solve equation \eqref{3.45}. Moreover, since $N(\rho)=O(|\rho|^{\min\{2,p\}})$, we deduce from Lemma \ref{reduction} that
\begin{equation}
\begin{aligned}
\|\rho\|_{H^1} &\approx  \|P_{\mathscr{F}^\perp} (-\Delta+1)^{-1} f\|_{H^{1}} +O(\|\rho\|_{H^1}^{\min\{2,p\}})\\&
\approx  \|P_{\mathscr{F}^\perp} (-\Delta+1)^{-1} f\|_{H^{1}}.
\end{aligned}
\end{equation}
\end{proof}

In general, the $H^{-1}$ norm,  is not always suitable for estimation purposes. However, for the function  $f$, we establish the following equivalence of norms:
\begin{lemma}\label{norm}
The norms  $\|f\|_{H^{-1}}, \|f\|_{L^2}$ and $\|f\|_{H^{1}}$  are equivalent for the function $f$ defined in \eqref{fh}.
\end{lemma}
\begin{proof}
Since \begin{equation}
\begin{aligned}
 |\nabla f |&= p\Big{|}\sum\limits_{i=1}^{m}(\sigma^{p-1}-Q_i^{p-1}) \nabla Q_i\Big{|}\\&
 \lesssim  \sum\limits_{i=1}^{m}(\sigma^{p-1}-Q_i^{p-1}) Q_i\\&
 \lesssim f
,
 \end{aligned}
 \end{equation}
we obtain the following estimates: $$ \|f\|_{L^2}\gtrsim \|f\|_{H^{1}}. $$
 Moreover, since
 $$ \|f\|_{H^{1}} \|f\|_{H^{-1}} \geq \|f\|_{L^2}^2  \mbox{ and } \|f\|_{H^{-1}} \leq \|f\|_{L^2} \leq \|f\|_{H^{1}},$$ combining all these estimates, we finish the proof.
\end{proof}

Consider the function $\rho$ obtained from Proposition \ref{rh},  
we have
\begin{equation}\|\rho\|_{H^1}\lesssim \|f\|_{L^2}\lesssim F_{d,p}(R^{-\frac{d-1}{2}}e^{-R}). \end{equation}
 Furthermore, the following proposition holds to be true:
\begin{proposition}\label{999}
  For $\rho$ obtained from Proposition \ref{rh}, the function $u=\sigma+\rho$ satisfies
  \begin{equation}
\|\rho\|_{H^1}\gtrsim F_{d,p}(\| (-\Delta+1) u- |u|^{p-1}u\|_{H^{-1}}).
\end{equation}
\end{proposition}
\begin{proof}
Direct calculation yields
\begin{equation}
\begin{aligned}
  (-\Delta+1) u- |u|^{p-1}u&=- (-\Delta+1)P_{\mathscr{F}}(-\Delta+1)^{-1} (f+N(\rho)) \\&+ p \sum\limits_{i,j}\int_{\mathbb{R}^d} \sigma^{p-1} \rho e_{ij} dx  (-\Delta+1) e_{ij}.
\end{aligned}
\end{equation}

Since $\{e_{ij}\}$ can be obtain by Schmidt orthogonalization of  \( \{ \partial_j Q_i \}_{ 1 \leq j \leq d, 1 \leq i \leq m} \), and $ \rho\in \mathscr{F}^\perp$,  as shown in equation \eqref{4.111} and \eqref{3.35}, we have
\begin{equation}
\begin{aligned}
\sum_{i,j } |\int_{\mathbb{R}^d} \sigma^{p-1} \rho e_{ij} dx| &\approx \sum_{i} |\int_{\mathbb{R}^d} (\sigma^{p-1}-Q_i^{p-1}) \rho \nabla Q_i dx|\\&
\lesssim
\|\rho\|_{H^1}^{p+1}+ e^{-(1-\frac{p(p-1)}{4}) R}\|\rho\|_{H^1}^{p-1}\\&
=o_R(R^{-\frac{d-1}{2}}e^{-R}).
\end{aligned}
\end{equation}
While  \eqref{3.36} implies
\begin{equation}
\begin{aligned}
 \|(-\Delta+1)P_{\mathscr{F}}(-\Delta+1)^{-1} &(f+N(\rho))\|_{H^{-1}} 
 = \|P_{\mathscr{F}}(-\Delta+1)^{-1} (f+N(\rho))\|_{H^{1}}\\&
 \approx \sum_{i} |\int_{\mathbb{R}^d} (f+N(\rho))\nabla Q_i dx|\\&
 \lesssim  \sum_{i} |\int_{\mathbb{R}^d}  f Q_idx|+\|\rho\|_{H^1}^2\\&
 \lesssim \sum\limits_{\substack{i,j \\ i\neq j}} |\int_{\mathbb{R}^d}  Q_i^p Q_j dx|+\|\rho\|_{H^1}^2\\&
  \approx   
R^{-\frac{d-1}{2}}e^{-R}.
\end{aligned}
\end{equation}
Therefore,  we must have
\begin{equation}
\begin{aligned}
\| (-\Delta+1) u- |u|^{p-1}u\|_{H^{-1}}
\lesssim R^{-\frac{d-1}{2}}e^{-R}.
\end{aligned}
\end{equation}
As a consequence,
\begin{equation}\label{sha}
\begin{aligned}
\|\rho\|_{H^1}&\gtrsim\| f\|_{H^{-1}}- O(R^{-\frac{d-1}{2}}e^{-R})\\& \gtrsim 
\| f\|_{L^2}- O(R^{-\frac{d-1}{2}}e^{-R})\\& \gtrsim 
F_{d,p}(R^{-\frac{d-1}{2}}e^{-R})\\&\gtrsim F_{d,p}(\| (-\Delta+1) u- |u|^{p-1}u\|_{H^{-1}}).
\end{aligned}
\end{equation}
\end{proof}
To finish the proof of Theorem \ref{sharp}, we still need to show:
\begin{lemma}\label{ex}
For $\rho$ obtained from Proposition \ref{rh}, the function $u=\sigma+\rho$ satisfies
  \begin{equation}
   \inf_{ \{x_k\} \subset \mathbb{R}^d} \|u-\sum_{k=1}^{m}Q(\cdot+x_k)\|_{H^1}\approx \|\rho\|_{H^1} .
  \end{equation}
\end{lemma}
\begin{proof}
Suppose there exists $u=\sigma'+\rho'$ such that $\|\rho'\|_{H^1}\leq \|\rho\|_{H^1}$.  
Then :
\begin{equation}\label{3.588}
\|\rho'\|_{H^1}^2=\|\sigma-\sigma'\|_{H^1}^2+\|\rho \|_{H^1}^2+2(\rho, \sigma-\sigma'  )_{H^1}.
\end{equation}
Without loss of generality, we can assume that   
there exist $\{x_i\}\subset \mathbb{R}^d$ with $|x_i|$ small enough, such that
$$\sigma'=\sum_{i=1}^{m}Q'_i(x),\quad Q'_i(x)=Q_i(x+x_i).$$ Note that
$$\|\sigma-\sigma'\|_{H^1}^2=\sum\limits_{i=1}^{m}\|Q_i-Q_i'\|_{H^1}^2 +\sum\limits_{\substack{i,j \\ i\neq j}}(Q_i-Q_i', Q_j-Q_j')_{H^1} .$$
Consider the smooth functions $$F_i(y):=\|Q_i(x)-Q_i(x+y)\|_{H^1}^2, G_{i,j}(y,z):= ( Q_i(x)-Q_i(x+y), Q_j(x)-Q_j(x+z) )_{H^1}.$$
By Taylor expansion of \( F_i \) and \( G_{i,j} \) at zero, we obtain the approximations:
\begin{equation*}
\quad\|Q_i-Q_i'\|_{H^1}^2 \approx |x_i|^2,   \text{~and~}  ( Q_i - Q'_i, Q_j - Q'_j)_{H^1} = o_{R}(1)(|x_i|^2 + |x_j|^2), ~~\forall i\neq j.
\end{equation*}
By the triangle inequality, we obtain the bound
\[ \sum\limits_{i=1}^{m}|x_i| \lesssim
\| \sigma - \sigma' \|_{H^1} \leq 2 \|\rho\|_{H^1}.
\]

Similarly, we have the following estimates:
\[
\| Q_i + x_i \cdot \nabla Q_i - Q_i' \|_{H^1}^2 \lesssim |x_i|^3.
\]
Since \( \rho \in \mathscr{F}^\perp \), 
 we have:
\[
\begin{aligned}
( \rho, \sigma - \sigma')_{H^1} &=(\rho, \sum\limits_{i=1}^{m}(Q_i + x_i \cdot \nabla Q_i - Q_i' ))_{H^1} \\
&\lesssim\|\rho\|_{H^1}^{2.5}.
\end{aligned}
\]
Thus, we obtain
$
\| \rho'\|_{H^1} \approx \|\rho\|_{H^1}
$
from \eqref{3.588}, this complete the proof.
\end{proof}

\begin{proof}[Proof of Theorem \ref{sharp}]
Combining Proposition \ref{999} and Lemma \ref{ex}, we deduce that the function $u$ satisfies the conditions and the conclusion stated in Theorem \ref{sharp}. Moreover, the function $u^+$ also continues to fulfill the conditions and the conclusion of Theorem \ref{sharp}. To verify this, recall that
\begin{equation}
\begin{aligned}
 (-\Delta+1) u- |u|^{p-1}u&=- (-\Delta+1)P_{\mathscr{F}}(-\Delta+1)^{-1} (f+N(\rho)) \\&+ p \sum\limits_{i,j}\int_{\mathbb{R}^d} \sigma^{p-1} \rho e_{ij} dx  (-\Delta+1) e_{ij}.
\end{aligned}
\end{equation}
Testing the equation with \( u^{-} \), by the estimates in the proof of Proposition \ref{999}, we obtain :
\[
\begin{aligned}
\|u^{-}\|_{H^1}^2\lesssim\| u^{-}\|_{H^1}^{p+1}+R^{-\frac{d-1}{2}}e^{-R}\| u^{-}\|_{H^1},
\end{aligned}
\]
thus $\|u^{-}\|_{H^1}\lesssim R^{-\frac{d-1}{2}}e^{-R}$.
Next, consider the difference:
\[
\begin{aligned}
 \Big{|}\| (\Delta - 1) u^+ &+ |u^+|^{p-1} u^+ \|_{H^{-1}} - \| (\Delta - 1) u + |u|^{p-1} u \|_{H^{-1}}  \Big{|}\\
& \leq \| u^{-} \|_{H^1} + \| |u^+|^{p-1} u^+ - |u|^{p-1} u \|_{H^{-1}} \\
& \lesssim \| u^{-} \|_{H^1} + \| u^{-} \|_{H^1}^p \\
& \lesssim R^{-\frac{d-1}{2}}e^{-R}.
\end{aligned}
\]
Therefore, 
\begin{equation}
\begin{aligned}
  \inf_{\{y_k\} \subset \mathbb{R}^d} \|u^+ - \sum_{k=1}^{m} Q(\cdot + y_k)\|_{H^1}& \gtrsim 
  F_{d, p}(R^{-\frac{d-1}{2}}e^{-R})-CR^{-\frac{d-1}{2}}e^{-R}\\& \gtrsim 
  F_{d, p}( \| (\Delta - 1) u^+ + |u^+|^{p-1} u^+\|_{H^{-1}}).
\end{aligned}
\end{equation}
\end{proof}

\subsection{Some Corollaries}

By the classification of positive solutions to \eqref{cfe} (cf. Kwong \cite{Kwong1989} ), we immediately obtain the following corollary for non-negative functions:

 \begin{corollary}\label{corollary1.1}
  For $d\geq 1$ and $1< p<2^*-1$, if $u\geq 0$ belongs to $H^1$  and satisfies $$(m-\frac{1}{2})\|Q\|_{H^1}^2\leq \|u\|_{H^1}^2\leq (m+\frac{1}{2})\|Q\|_{H^1}^2,$$
   then for some  constant $C(m, d, p)$ depending on $m, d$ and $p$, the following estimate holds:
  \[    \inf_{\{y_k\}\subset \mathbb{R}^d} \|u-\sum_{k=1}^{m}Q(\cdot+y_k)\|_{H^1} \leq C(m, d, p)F_{d,p}(\|\Delta u-u+|u|^{p-1}u\|_{H^{-1}}).  \]
\end{corollary}
In the one-dimensional case,  the ground state is given by $$Q(x)=\Big{(}\frac{p+1}{2\cosh^2(\frac{p-1}{2} x)}\Big{)}^{\frac{1}{p-1}},$$ 
 according to Cazenave \cite[Theorem 8.16]{Cazenave2003}, we have the following corollary:
\begin{corollary}\label{corollary1.2}
 For $d=1$ and $p>1$,  if $u\in H^1$ satisfies $$(m-\frac{1}{2})\|\Big{(}\frac{p+1}{2\cosh^2(\frac{p-1}{2} x)}\Big{)}^{\frac{1}{p-1}}\|_{H^1}^2\leq \|u\|_{H^1}^2\leq (m+\frac{1}{2})\|\Big{(}\frac{p+1}{2\cosh^2(\frac{p-1}{2} x)}\Big{)}^{\frac{1}{p-1}}\|_{H^1}^2,$$
   then for some  constant $C(m,p)$ depending on $m$ and $p$, the following estimate holds:
  \[    \inf_{\{y_k\}\subset  \mathbb{R}} \|u-\sum_{k=1}^{m}Q(\cdot+y_k)\|_{H^1} \leq C(m, p)F_{1,p}(\|\Delta u-u+|u|^{p-1}u\|_{H^{-1}}).  \]
\end{corollary}

\begin{proof}[Proof of Corollary \ref{corollary1.1}]
 This proof is standard. Suppose, on the contrary, that for each $n$, there exists non-negative function $u_n$, such that $$(m-\frac{1}{2})\|Q\|_{H^1}^2\leq \|u_n\|_{H^1}^2\leq (m+\frac{1}{2})\|Q\|_{H^1}^2,$$ and
\begin{equation}\label{mddd}
  \inf_{ \{y_k\} \subset \mathbb{R}^d} \|u_n-\sum_{k=1}^{m}Q(\cdot+y_k)\|_{H^1} > nF_{d,p}(\|\Delta u_n-u_n+|u_n|^{p-1}u_n\|_{H^{-1}}).
\end{equation}
 Thus, $\{u_n\}$ is a Palais–Smale  sequence for the functional associated with  the Euler-Lagrange  equation $\Delta u-u+|u|^{p-1}u$. Theorem \ref{theorem1.3} yields that, up to a subsequence,
 there exist  a sequence $\{v^{k}\}_{k=1}^{m}$ such that $$\Delta v^{k}-v^{k}+|v^{k}|^{p-1}v^{k}=0,$$    and a sequence of points $y_{n,k}$ such that
   \[\min_{i\neq j}|y_{n,i}-y_{n,j}| \to +\infty,\]
and
   $$ u_n \to \sum_{k=1}^{m} v^{k}(\cdot+y_{n,k})\text{~with } \|u_{n}\|_{H^1}^2 \to \sum_{k=1}^{m} \|v^{k}\|_{H^1}^2\text{~as~} n\to+\infty .$$
 In particular, since $u_n$ are non-negative, and $u_{n}(\cdot-y_{n,k})\rightharpoonup  v^{k}$, it follows that $ v^{k}$  are non-negative. By the well-known uniqueness result of
 \cite{Kwong1989}, we have $v^{k}=Q(\cdot+y_{k})$ for some $y_{k}\in \mathbb{R}^d$.
 
Therefore, the assumption in \eqref{mddd} contradicts Theorem \ref{theorem1.2}.

\end{proof}

\begin{proof}[Proof of Corollary \ref{corollary1.2} ]
The proof is similar to that of Corollary \ref{corollary1.1}. The only difference is that in one dimension, every real-valued solution to the equation  $$\frac{d^2u}{d x^2} -u+|u|^{p-1}u=0$$
is of the form $Q(\cdot+y)$ for some $y\in \mathbb{R}$, as shown in \cite[Theorem 8.16]{Cazenave2003}. Therefore, we do not require the solution to be non-negative.
\end{proof}
\section{The complex-valued case}
The problem becomes more complicated in the complex-valued case. Even if $p>2$, we still need to add restrictions on  the phase difference. It remains uncertain whether the restriction can be removed.

\subsection{Finite number of ground states with cubic nonlinearity}
 To begin with, we consider the cubic nonlinearity,  i.e., $p=3$.
 Specifically, we have the following result:
\begin{theorem}\label{theorem1.4}
  Let $4>d\geq 1$ and   $p=3$. For each $m$ and $c\in (0, 1]$, there exists $\widetilde{R}=\widetilde{R}(m,c,p)>0$  such that
  if $u\in H^{1}(\mathbb{R}^d; \mathbb{C})$ satisfies
   $$\|u-\sum_{k=1}^{m}e^{i\theta_k}Q(\cdot+\widetilde{y}_k)\|_{H^1}\leq \frac{1}{\widetilde{R}},~\text{for ~some~} \{\widetilde{y}_k\}\subset \mathbb{R}^d \text{~with~} \min_{i\neq j}|\widetilde{y}_i-\widetilde{y}_j|\geq \widetilde{R},$$
and  some $\theta_i\in [0, 2\pi)$ satisfying the phase restriction :
  \begin{equation}\label{arg}
             \cos|\theta_{i} -\theta_{j}|>c,~ \forall i\neq j \text{~~or~}  \cos|\theta_{i} -\theta_{j}|<-c,~ \forall i\neq j ,
           \end{equation}
then for some  constant $C(c,m,p)$ depending on $c, m$ and $p$, the following estimate holds:
 \begin{equation}\label{cpl}
 \inf_{\{y_k\}\subset \mathbb{R}^d,\theta_k\in \mathbb{R}} \|u-\sum_{k=1}^{m}e^{i\theta_k}Q(\cdot+y_k)\|_{H^1} \leq C(c,m,p)\|\Delta u-u+|u|^{2}u\|_{H^{-1}}.
 \end{equation}
  
\end{theorem}
\begin{proof}[Proof of Theorem \ref{theorem1.4}]
\hfill
\par
\emph{Step 1:}
By the assumption of Theorem \ref{theorem1.4}, there exist \( z_i \in \mathbb{C} \) and \( y_i \in \mathbb{R}^d \) such that
\[
\sum_{i=1}^{m} z_i Q_i: = \sum_{i=1}^{m} z_i Q(\cdot + y_i) 
\]
achieves
\[
\inf_{\{z_i\} \subset \mathbb{C}^d, \{y_i\} \subset \mathbb{R}^d} \| u - \sum_{i=1}^{m} z_i Q(\cdot + y_i) \|_{H^1}.
\]
In particular, it holds that
\[
|z_i| = 1 + o_{\widetilde{R}}(1), \quad  R:=\min_{i \neq j} |y_i - y_j| \geq \frac{\widetilde{R}}{2},
\]
and either
\[
\text{Re}(z_j \bar{z}_i) > \frac{c}{2} \quad \text{for all } i \neq j, 
\]
or
\[
\text{Re}(z_j \bar{z}_i) < -\frac{c}{2} \quad \text{for all } i \neq j.
\]
Furthermore, the function \( \rho := u - \sum\limits_{i=1}^{m} z_i Q_i \) satisfies the following orthogonality conditions:
\[
\int_{\mathbb{R}^d}  (\rho Q_i + \nabla \rho \cdot \nabla Q_i)  dx= 0, \quad \int_{\mathbb{R}^d}  \rho Q_i^3  dx= 0, \quad \text{Re} \int_{\mathbb{R}^d}  \rho \bar{z_i} Q_i^2 \nabla Q_i dx = 0.
\]
By direct calculations, we have
\begin{equation}\label{com}
\begin{aligned}
\Delta u-&u+|u|^{2}u=(\Delta-1) \rho-\sum_{i=1}^{m}z_iQ_i^3+|\sum_{i=1}^{m} z_iQ_i+\rho|^2(\sum_{i=1}^{m} z_iQ_i+\rho)\\&
=(\Delta-1) \rho+\sum_{i=1}^{m}(|z_i|^2-1)z_iQ_i^3+\sum\limits_{\substack{i,j \\ i\neq j}}|z_i|^2z_jQ_i^2Q_j+\sum\limits_{k=1}^{m}\sum\limits_{\substack{i,j \\ i\neq j}}\text{Re}( z_i \bar{z}_j) z_k Q_iQ_jQ_k\\&
+\sum_{i=1}^{m}|z_i|^2Q_i^2\rho+\sum\limits_{\substack{i,j \\ i\neq j}}\text{Re}( z_i \bar{z}_j) Q_iQ_j\rho+\sum_{i,j}2\text{Re}(\rho \bar{z}_i)z_jQ_iQ_j+O(\rho^2).
\end{aligned}
\end{equation}
By testing the equation with \( \bar{\rho} \) and taking the real part, we obtain:
\begin{equation}\label{5.2}
\begin{aligned}
\|\rho\|_{H^1}^2&\lesssim \|\Delta u-u+|u|^{2}u\|_{H^{-1}}\|\rho\|_{H^1} +o_{\widetilde{R}}(\rho^2) +\sum_{i=1}^{m}\int_{\mathbb{R}^d} |z_i|^2Q_i^2|\rho|^2 dx\\& + \sum\limits_{\substack{i,j \\ i\neq j}} |y_{i}-y_{j}|^{-\frac{d-1}{2}}e^{- |y_{i}-y_{j}|}  \|\rho\|_{H^1} +\sum_{i=1}^{m} 2|\text{Re}(\rho \bar{z}_i)|^2 Q_i^2.
\end{aligned}
\end{equation}

\par
\emph{Step 2:}
We claim that 

\begin{equation}\label{5.333}
\begin{aligned}
\sum_{i=1}^{m}\int_{\mathbb{R}^d}  |z_i|^2Q_i^2|\rho|^2 dx&+ \sum_{i=1}^{m} 2|\text{Re}(\rho \bar{z_i})|^2 Q_i^2\\&
=\sum_{i=1}^{m} \int_{\mathbb{R}^d}  3|\text{Re}(\rho \bar{z_i})|^2 Q_i^2 + |\text{Im}(\rho \bar{z_i})|^2 Q_i^2
\\&
\leq \frac{3+o_{\widetilde{R}}(1)}{3+\kappa} \|\rho\|_{H^1}^2.
\end{aligned}
\end{equation}
Therefore, we deduce from \eqref{5.2} that
\begin{equation}\label{5.4}
  \|\rho\|_{H^1}\lesssim \|\Delta u-u+|u|^{2}u\|_{H^{-1}}+R^{-\frac{(d-1)}{2}}e^{-R}.
\end{equation}

In fact, by consider the cut-off function  $\Phi_i =1$ in $B(-y_i,\frac{\sqrt{R}}{3})$, and $\Phi_i=0$ in $B(-y_i,\frac{\sqrt{R}}{2})^c$. We have
\begin{equation}\label{3-5}
\begin{aligned}
  \sum_{i=1}^{m}\int_{\mathbb{R}^d}  |\text{Re}(\rho \bar{z_i})|^2 Q_i^2 dx &\leq     \sum_{i=1}^{m}\int_{\mathbb{R}^d}  |\Phi_i|^2Q_i^2|\text{Re}(\rho \bar{z_i})|^2 dx +o_{\widetilde{R}}(1)\|\rho\|_{H^1}^2
  .
  \end{aligned}
  \end{equation}
We deduce from \eqref{gap} and the orthogonality conditions that
\begin{equation}
\begin{aligned}
\|\Phi_i \text{Re}(\rho \bar{z_i})\|_{H^1}^2&\geq (3+\kappa) \int_{\mathbb{R}^d }  Q_i^{2} |\Phi_i\text{Re}(\rho \bar{z_i})|^2 dx\\&- C\Big{|}\int_{\mathbb{R}^d }  Q_{i}^{3}(1-\Phi_i)\text{Re}(\rho \bar{z_i}) dx\Big{|}^2 - C\Big{|}\int_{\mathbb{R}^d }  Q_{i}^{2}\nabla Q_{i}(1-\Phi_i)\text{Re}(\rho \bar{z_i}) dx\Big{|}^2
\\
&
\geq (3+\kappa) \int_{\mathbb{R}^d } \Phi_i^2 Q_i^{2} |\text{Re}(\rho \bar{z_i})|^2 dx-o_{\widetilde{R}}(1)\|\rho\|_{H^1}^2.
\end{aligned}
\end{equation}
Since $|\nabla \Phi_i|\leq \frac{C}{\sqrt{R}}= o_{\widetilde{R}}(1)$,   easy to check
\begin{equation}\label{3-9}
\sum_{i=1}^{m}\|\Phi_i \text{Re}(\rho \bar{z_i})\|_{H^1}^2
 \le\sum_{i=1}^{m} \int_{\mathbb{R}^d } (|\nabla \text{Re}(\rho \bar{z_i})|^2 \Phi_i^2 +|\text{Re}(\rho \bar{z_i})|^2 \Phi_i^2)dx+o_{\widetilde{R}}(1)\|\rho\|_{H^1}^2.
\end{equation}
As the various  functions $\Phi_i $ have disjoint supports, therefore
\begin{equation}\label{3-10}
\sum_{i=1}^m \int_{\mathbb{R}^d } |\nabla \text{Re}(\rho \bar{z_i})|^2 \Phi_i^2 dx +\sum_{i=1}^m \int_{\mathbb{R}^d } | \text{Re}(\rho \bar{z_i})|^2 \Phi_i^2 dx\le \|\text{Re}(\rho \bar{z_i})\|_{H^1}^2.
\end{equation}
Thus, we can conclude from \eqref{3-5}-\eqref{3-10} that
\begin{align}\label{5.3333}
  \sum_{i=1}^{m}\int_{\mathbb{R}^d}  |\text{Re}(\rho \bar{z_i})|^2 Q_i^2 dx \le  \frac{1}{3+\kappa}\|\text{Re}(\rho \bar{z_i})\|_{H^1}^2 +o_{\widetilde{R}}(1)\|\rho\|_{H^1}^2.
\end{align}

As for the term 
$\sum\limits_{i=1}^{m} \int_{\mathbb{R}^d}  |\text{Im}(\rho \bar{z_i})|^2 Q_i^2 dx$, we deduce from
\eqref{gap} and the orthogonality conditions that
\begin{equation}
\begin{aligned}
\|\Phi_i \text{Im}(\rho \bar{z_i})\|_{H^1}^2&\geq 3 \int_{\mathbb{R}^d }  Q_i^{2} |\Phi_i\text{Im}(\rho \bar{z_i})|^2 dx- C\Big{|}\int_{\mathbb{R}^d }  Q_{i}^{3}(1-\Phi_i)\text{Im}(\rho \bar{z_i}) dx\Big{|}^2 
\\
&
\geq 3 \int_{\mathbb{R}^d } \Phi_i^2 Q_i^{2} |\text{Im}(\rho \bar{z_i})|^2 dx-o_{\widetilde{R}}(1)\|\rho\|_{H^1}^2.
\end{aligned}
\end{equation}
Similar arguments yields that 
\begin{equation}\label{5.33333}
\sum\limits_{i=1}^{m} \int_{\mathbb{R}^d}  |\text{Im}(\rho \bar{z_i})|^2 Q_i^2  dx\le  \frac{1}{3}\|\text{Im}(\rho \bar{z_i})\|_{H^1}^2 +o_{\widetilde{R}}(1)\|\rho\|_{H^1}^2.
\end{equation}
And then \eqref{5.333} follows from  \eqref{5.3333} ,\eqref{5.33333} and the fact $|z_i| = 1 + o_{\widetilde{R}}(1)$.

\par
\emph{Step 3:}
Test the equation \eqref{com} by $Q_{i_0}$   we obtain:
\begin{equation}\label{5.11}
\begin{aligned}
|(|z_{i_0}|^2&-1)z_{i_0}\int_{\mathbb{R}^d} Q_{i_0}^4 dx+\sum_{j \neq {i_0}}(|z_{i_0}|^2z_j +2\text{Re}(z_{i_0}z_j) z_{i_0} )\int_{\mathbb{R}^d}  Q_{i_0}^3Q_j dx|\\
\lesssim &\|\Delta u-u+|u|^{2}u\|_{H^{-1}}+o_R(1) \|\rho\|_{H^1}+ o_R(R_{i_0}^{-\frac{d-1}{2}}e^{-R_{i_0}} ).
\end{aligned}
\end{equation}
Test the equation  \eqref{com} by $\bar{z}_{i_0}\nabla Q_{i_0}$ and take the real part  we obtain:
\begin{equation}\label{5.12}
\begin{aligned}
|3|z_{i_0}|^2\sum_{j\neq i_0}\int_{\mathbb{R}^d} & \text{Re}(z_j \bar{z}_{i_0})Q_{i_0}^2 Q_{j}\nabla Q_{i_0} dx|
\\&\lesssim \|\Delta u-u+|u|^{2}u\|_{H^{-1}}+o_R(1)\|\rho\|_{H^1} +o_R(R_{i_0}^{-\frac{d-1}{2}}e^{-R_{i_0}} ).
\end{aligned}
\end{equation}

Since
all the $\{\text{Re}(z_j \bar{z}_{i}) \}_{i\neq j}$ have the same sign, and $|\text{Re}(z_j \bar{z}_{i})|>\frac{c}{2}$, by Lemma \ref{point} and \ref{lemmay}, we may assume that 
\begin{equation}\label{5.13}
\begin{aligned}
R_1^{-\frac{(d-1)}{2}}e^{-R_1}&\lesssim
    \Big{|}\sum\limits_{j>1} \frac{y_{j}^{1}-y_{1}^1}{|y_{j}-y_{1}|} |y_{j}-y_{1}|^{-\frac{d-1}{2}} e^{-|y_{j}-y_{1}|} \Big{|}\\&\lesssim  |\sum_{j>1}\int_{\mathbb{R}^d} \text{Re}(z_j \bar{z}_{1})Q_{1}^2 Q_{j}\partial_1 Q_{1} dx|
\\&\lesssim \|\Delta u-u+|u|^{2}u\|_{H^{-1}}+o_R(1)\|\rho\|_{H^1} .
\end{aligned}
\end{equation}.

Then \eqref{5.11} yields that
\[  ||z_{1}|-1|\lesssim \|\Delta u-u+|u|^{2}u\|_{H^{-1}}+o_R(1)\|\rho\|_{H^1} . \]

By induction we obtain that
\begin{equation}\label{5.15}
R^{-\frac{(d-1)}{2}}e^{-R}\lesssim
    \|\Delta u-u+|u|^{2}u\|_{H^{-1}}+o_R(1)\|\rho\|_{H^1} .
\end{equation}
And
\begin{equation}\label{5.16}
\max_{i}||z_{i}|-1|\lesssim
    \|\Delta u-u+|u|^{2}u\|_{H^{-1}}+o_R(1)\|\rho\|_{H^1} .
\end{equation}

Therefore, we deduce from \eqref{5.4}, \eqref{5.15} and \eqref{5.16} that
\[ \max_{i}||z_{i}|-1|+ \|\rho\|_{H^1} \lesssim  \|\Delta u-u+|u|^{2}u\|_{H^{-1}}.\]
As a consequence, 
 \begin{equation}
 \inf_{\{y_k\}\subset \mathbb{R}^d,\theta_k\in \mathbb{R}} \|u-\sum_{k=1}^{m}e^{i\theta_k}Q(\cdot+y_k)\|_{H^1} \leq C(c,m,p)\|\Delta u-u+|u|^{2}u\|_{H^{-1}}.
 \end{equation}
\end{proof}

\subsection{A single ground state}
We emphasize that the restrictions on the phase difference and the value of
 $p,d$ are due to the interaction between different ground states. In the case of a single ground state, i.e.  $m=1$, the situation is much simpler:
\begin{theorem}\label{theorem1.5}
  Let $d\geq 1$ and $2^*-1>p>1$. There exists $\epsilon=\epsilon(p)>0$  such that
  if $u\in H^{1}(\mathbb{R}^d; \mathbb{C})$ satisfies
   $$\|u-e^{i\theta}Q(\cdot+\widetilde{y})\|_{H^1}\leq \epsilon
 \mbox{ for some } \theta\in [0, 2\pi) \mbox{ and } \widetilde{y} \in \mathbb{R}^d,$$
then we have the following estimate:
 \begin{equation}
 \inf_{y\in \mathbb{R}^d,\theta\in \mathbb{R}} \|u-e^{i\theta}Q(\cdot+y)\|_{H^1} \leq C(d, p)\|\Delta u-u+|u|^{p-1}u\|_{H^{-1}}.
 \end{equation}
\end{theorem}

As for the application to the  NLS equations, by the well-known result on the orbital stability of a single soliton wave in the $L^2$ subcritical case (cf. Weinstein \cite{Weinstein1985} and \cite{Weinstein1986}), we obtain the following corollary:
  \begin{corollary}\label{corollary1.3}
  For  $5>p>1$, $d=1$. Let $u_0\in H^{1}(\mathbb{R}; \mathbb{C})$ with $$\frac{1}{2}\|\Big{(}\frac{p+1}{2\cosh^2(\frac{p-1}{2} x)}\Big{)}^{\frac{1}{p-1}}\|_{H^1}^2\leq \|u_0\|_{H^1}^2\leq \frac{3}{2}\|\Big{(}\frac{p+1}{2\cosh^2(\frac{p-1}{2} x)}\Big{)}^{\frac{1}{p-1}}\|_{H^1}^2.$$ Then $u(t)$, the solution to the NLS equation with initial value
 $u_0$, satisfies
    \[    \inf_{y, \theta\in \mathbb{R}} \|u(t)-e^{i\theta} Q(\cdot+y)\|_{H^1} \leq C\|\Delta u_0-u_0+|u_0|^{p-1}u_0\|_{H^{-1}}.  \]
  \end{corollary}
\begin{proof}[Proof of  Theorem \ref{theorem1.5}]
Without loss of generality, up to a translation and  phase rotation,  we can  write $u=\alpha Q+\rho$, such that
$$\alpha \in\mathbb{R},~|\alpha-1|\lesssim \epsilon , ~\rho \in H^{1}(\mathbb{R}^d, \mathbb{C}) ~\text{with~}\|\rho\|_{H^1}\leq \epsilon,$$  and $\rho$ satisfies the orthogonality condition:
\begin{equation}\label{zj}
 \int_{\mathbb{R}^d}  (\rho Q + \nabla \rho \cdot \nabla Q) dx =\int_{\mathbb{R}^d}   \rho Q^p dx =0,~~  \text{Re} \int_{\mathbb{R}^d}   \rho   Q^{p-1} \nabla Q  dx=0.
\end{equation}
Let $\rho=\rho_1+i\rho_2$,  direct calculations yield
\begin{equation}\label{5.81}
\begin{aligned}
-\Delta u+u-|u|^{p-1}u&=-\Delta \rho+\rho-
p|\alpha Q|^{p-1}\rho_1-i|\alpha Q|^{p-1}\rho_2+(\alpha-\alpha^p) Q^p\\&+
\Big{[}
\alpha^p Q^p+p|\alpha Q|^{p-1}\rho_1+i|\alpha Q|^{p-1}\rho_2-|\alpha Q+\rho|^{p-1}(\alpha Q+\rho) \Big{]}.
\end{aligned}
\end{equation}
By applying a Taylor expansion, we have the following approximation:
\[ \Big{|}
\alpha^p Q^p+p|\alpha Q|^{p-1}\rho_1+i|\alpha Q|^{p-1}\rho_2-|\alpha Q+\rho|^{p-1}(\alpha Q+\rho) \Big{|}\lesssim |\rho|^{\min\{2,p\}}. \]
 Therefore, testing \eqref{5.81} with $\bar{\rho}$ and taking the real part yields:
\begin{equation}
\begin{aligned}
\|\rho\|_{H^1}^2 &\leq  \int_{\mathbb{R}^d} p|\alpha Q|^{p-1}\rho_1^2 dx+ \int_{\mathbb{R}^d}|\alpha Q|^{p-1}\rho_2^2dx +  C\|\rho\|_{H^1}^{\min\{3,p+1\}}\\&+  \|-\Delta u+u-|u|^{p-1}u\|_{H^{-1}}\|\rho\|_{H^1}.
\end{aligned}
\end{equation}
From \eqref{zj} and \eqref{gap}, we obtain:
\[ (p+\kappa)\int_{\mathbb{R}^d} Q^{p-1}\rho_1^2  dx\leq \|\rho_1\|_{H^1}^2, \quad  p \int_{\mathbb{R}^d} Q^{p-1}\rho_2^2 dx\leq \|\rho_2\|_{H^1}^2 . \]
Thus,  for sufficiently small $\epsilon$, we have
\[ \|\rho\|_{H^1}\lesssim \|-\Delta u+u-|u|^{p-1}u\|_{H^{-1}}. \]
Testing \eqref{5.81} with $Q$, we obtain
\[ |\alpha-1|\lesssim \|-\Delta u+u-|u|^{p-1}u\|_{H^{-1}}. \]
Thus the conclusion follows.
\end{proof}

\begin{proof}[Proof of Corollary \ref{corollary1.3}]
Note that the co-compact result in Theorem \ref{theorem1.3} still holds for complex-valued functions. 
  From the proof of Corollary \ref{corollary1.2}, we actually have
    \[    \inf_{y, \theta\in \mathbb{R}} \|u_0-e^{i\theta} Q(\cdot+y)\|_{H^1} \leq C\|\Delta u_0-u_0+|u_0|^{p-1}u_0\|_{H^{-1}}.  \]
    Moreover, all complex-valued solutions to the equation
  $$\frac{d^2u}{d x^2} -u+|u|^{p-1}u=0$$
are of the form $e^{i\theta}Q(\cdot+y)$ for some $y,\theta \in \mathbb{R}$ ( cf. \cite[Theorem 8.16]{Cazenave2003}). Therefore, Corollary \ref{corollary1.3} follows directly from the conservation of energy, and the results of Weinstein \cite{Weinstein1985} and \cite{Weinstein1986} (also see \cite[Section 2.3]{Martel2006}):
In $L^2$ subcritical case, i.e. when $p<1+\frac{4}{d}$,  the solution $u(t)$ to the NLS  equation with initial value
 $u_0$ satisfies
    \[    \inf_{y, \theta\in \mathbb{R}} \|u(t)-e^{i\theta} Q(\cdot+y)\|_{H^1} \leq C\inf_{y, \theta\in \mathbb{R}} \|u_0-e^{i\theta} Q(\cdot+y)\|_{H^1},  \]
   provided that $\inf\limits_{y, \theta\in \mathbb{R}} \|u_0-e^{i\theta} Q(\cdot+y)\|_{H^1}$  is sufficiently small.
\end{proof}

\section{Acknowledgements}
Authors want to give their special thanks to Ilya Bogdanov for providing valuable insights on Lemma \ref{point} via MathOverflow. Also, they are grateful to Giorgio Metafune and Iosif Pinelis for helpful comments regarding the questions on the ODE of the ground state. Additionally, Xin Liao extends his sincere gratitude to Professors Tai-Peng Tsai and Herbert Koch for their insightful responses to the inquiries on soliton theory.
Hua Chen is supported by National Natural Science Foundation of China (Grant Nos. 12131017, 12221001) and National Key R$\&$D Program of China (no. 2022YFA1005602).
\section{Declarations}
On behalf of all authors, the corresponding author states that there is no conflict of interest. Our manuscript has no associated data.

\end{document}